\documentclass[reqno, a4paper, 11pt]{book}

\usepackage[francais]{projet}
\usepackage{esint}
\usepackage{dsfont}
\usepackage{graphicx}
\usepackage{enumerate}
\usepackage{amsmath}
\usepackage[ansinew]{inputenc}
\usepackage{textcomp}
\usepackage[all]{xy}
\usepackage{makeidx}

\makeindex

\begin{document}

\pagenumbering{Roman}

\thispagestyle{empty}

\begin{titlepage}
\quad \vspace{4.5cm}
\begin{center}
\huge \bfseries Topological amenability \\
\vspace{0.5cm} \mdseries \large António Terra\\
\vspace{2cm}
\large Mathematics Section\\Master Thesis\\Autumn 2010\\
\vspace{7cm}
\begin{tabular}{lcr}
Supervisor : Professor Nicolas Monod
\end{tabular}
\end{center}
\end{titlepage}
\frontmatter
\tableofcontents
\chapter{Introduction}
This document is the report of a Master Thesis in Mathematics carried out during the semester of autumn 2010 under the supervision of Prof. Nicolas Monod. The work comprised two parts. The first one was to read and to apprehend the details of Monod's paper~\cite{AVENIR}, skipping the parts about bounded cohomology. Omitting the cohomological aspects, the main result of this paper is the following.
\begin{thm}
 Let $G$ be a group acting on a compact space X. The following are equivalent.
\begin{itemize}
 \item[(i)]The action of $G$ on $X$ is topologically amenable.
 \item[(ii)]Every dual $(G, X)$-module of type C is a relatively injective Banach $G$-module.
 \item[(iii)]There is a $G$-invariant element in $C(X, \ell^1G)^{**}$ summing to $\indicc{X}$.
 \item[(iv)]There is a norm one positive $G$-invariant element in $C(X,\ell^1G)^{**}$ summing to $\indicc{X}$.
\end{itemize}
\end{thm}

In the theorem above, the group $G$ is endowed with the discrete topology. The second part of the work was to formulate and prove an equivalent result for a general locally compact topology on $G$. It has been achieved by replacing every occurrence of $\ell^1G$ in the above statement by $L^1G$. The precise statement and the proof of this result is the object of Section~\ref{result:Sect}, Chapter~\ref{result:Ch}. For readers familiar to cohomology, an extended result including cohomological aspects is stated and briefly proved at the end of this introduction. See also~\cite{Autres} for related results, excluding however relative injectivity.

This report is written in such a way that any student finishing mathematics studies can read and understand almost everything. In order to reach this goal, the reader is introduced to some non standard topics. This material represents the first three chapters of the document. From the self-containedness point of view, this report suffers from two weaknesses. The biggest one is the lack of a chapter devoted to locally compact topological groups. However, only basic facts (that can easily be found in the abundant literature about analysis on groups) are used. The second weakness is the use of an identification of normed vector spaces in the proof of Theorem~\ref{main_result}. More precisely, the fact that for any Banach spaces $B$ and $V$ the space of bounded linear maps from $B$ to $V^*$ is isometrically isomorphic to the dual of the projective tensor product of $B$ and $V$ is used. In other words, we use the identification $\mathcal{L}(B, V^*)\cong (B \hat{\otimes} V)^*$. Since normed tensor products occur in several arguments in~\cite{AVENIR}, the topic has been studied during the semester. However, in the general locally compact setting, the proofs avoid almost totally tensor products and therefore adding an entire chapter about them seemed disproportionate.

Before briefly explaining the content of the subsequent chapters, some categorical facts should be mentioned. They may enlighten the central definition of this work (that is to say the definition of an \emph{amenable transformation group}) and lead to one of the main reasons for the study of amenable transformation groups. Let $G$ be a locally compact topological group. The category of compact $G$-spaces has both an initial and a terminal object, namely the Stone-Cech compactification of the group $\beta G$ and the one point set $\{\bullet\}$. If $X$ and $Y$ are two compact $G$-spaces with an arrow from $X$ to $Y$, then the amenability of the transformation group $(G, Y)$ implies the amenability of $(G, X)$ via the following diagram.
$$
\xymatrix{ X \ar[rr] \ar@{.>}[rd] && Y \ar[ld] \\ & \mathrm{Prob}(G)}
$$
Thus, the amenability at the terminal object of the category implies that every transformation group is amenable. This also implies that the amenability anywhere in the category implies the amenability of the transformation group $(G, \beta G)$, since every compact $G$-space is the end of an arrow starting at $\beta G$. From that point of view, the definition of an amenable transformation group is not confusing anymore, since it really gives information about the group $G$ itself (through $\beta G$).

Chapter~\ref{lattices:Ch} is an introduction to the theory of Banach lattices. We start with general vector lattices (which are nothing but vector spaces endowed with a compatible lattice order), give the basic computational tools and properties and quickly move to normed lattices. The latter are vector lattices whose underlying vector spaces is endowed is an order compatible norm. Asking these normed spaces to be complete, we obtain Banach lattices. The theory is pursued up to basic facts about dual and double dual spaces of Banach lattices. In particular, it is shown how to endow dual spaces with vector lattices structures and that the canonical embedding of a Banach space into its double dual is a Banach lattice isomorphism.

Chapter~\ref{integral:Ch} is devoted to the Bochner integral. This integral allows to integrate maps from a measure space into a Banach space. The approach is quite usual for that kind of topic. In the first section of the chapter, integrable maps are identified and in the second one, their integral is defined. The last section gives the main properties about the Bochner integral : commutativity with bounded linear maps, an estimate about the norm of an integral, dominated convergence and Fubini's theorems.

Chapter~\ref{modules:Ch} finally gets closer to the main subject of the thesis, since it introduces the structures and spaces that occur in the statement of the main result (Theorem~\ref{main_result}). Let $G$ be a topological group and $X$ a continuous compact $G$-space. The goal of the first section of Chapter~\ref{modules:Ch} is to show, after having defined it, that the diagonal action of $G$ on $C(X, L^1G)$ is continuous. The second section is mainly a collection of definitions, among which the definitions of Banach $G$-modules, $(G, X)$-modules and relative injectivity $G$-modules. In a paper from 1980, Anker proved in an elegant way the equivalence between Properties $(P_1)$ and $(P_1^*)$ (see Section~\ref{Anker:Sect} for definitions), which ensure the existence in $L^1G$ of $\epsilon$-invariant vectors under the action of compacts, respectively finite, subsets of $G$. The last section of Chapter~\ref{modules:Ch} defines similar properties for the space $C(X, L^1G)$ and shows that they are also equivalent.

Chapter~\ref{result:Ch} is the reason of all the other chapters.First, there is a precise definition of what an amenable transformation group is. Then, an equivalent definition (that will be the one used in the proof of Theorem~\ref{main_result}) is given. Roughly speaking, the equivalence is based on the fact that the probability measures that occur in the definition of amenability can be taken to be density measures with respect to the Haar measure of the group. Then follows a miscellanea of auxiliary results that are useful in the final proof. The last section of the report is the statement and the proof of the result stated above for a general locally compact topology on the group $G$ and with $L^1G$ instead of $\ell^1G$.

For the sake of self-containedness, Theorem~\ref{main_result} does not mention the link with cohomology. However, for readers used to the cohomological language, we briefly explain here how to obtain an extended theorem that includes the cohomological issue. With the assumption of Chapter~\ref{result:Ch} on the spaces $G$ and $X$, the extended result is the following.

\begin{thm}\label{ext_main_result}
 Let $(G, X)$ be a transformation group, where $X$ is compact. The following are equivalent.
\begin{itemize}
 \item[(a)]$(G, X)$ is an amenable transformation group.
  \item[(b)]Every dual $(G, X)$-module of type C is a relatively injective Banach $G$-module.
  \item[(c)]$\mathrm{H}^n_\mathrm{b}(G, E^*)= 0$ for every $(G, X)$-module $E$ of type M and every $n \geq 1$.
  \item[(d)]There is $G$-invariant element in $C(X, L^1G)^{**}$ summing to $\indicc{X}$ in $C(X)^{**}$.
  \item[(e)]There is a norm one positive $G$-invariant element in $C(X, L^1G)^{**}$ summing to $\indicc{X}$ in $C(X)^{**}$.
\end{itemize}
\end{thm}

The strategy of the proof is (a) $\Rightarrow$ (b) $\Rightarrow$ (c) $\Rightarrow$ (d) $\Rightarrow$ (a) and (a) $\Leftrightarrow$ (e). According to Theorem~\ref{main_result} it suffices to show the implications (b) $\Rightarrow$ (c) $\Rightarrow$ (d) to obtain the above theorem. (A direct proof of (b) $\Rightarrow$ (d) is available in Chapter~\ref{result:Ch}) That (b) implies (c) is due to the fact that cohomology with values in relatively
injective modules vanishes in every positive degree (e.g. because the module is a trivial resolution
of itself. See~\cite{These_Monod}, Proposition 7.4.1.). To show that (c) implies (d), consider the short exact sequence
$$
0 \longrightarrow C(X, L^1_0G) \overset{\iota}{\longrightarrow} C(X, L^1G) \overset{I}{\longrightarrow} C(X) \longrightarrow 0,
$$
where $I$ is the continuous linear map that sends $f$ in $C(X, L^1G)$ to $x \mapsto \int_G f(x, t) \, d\lambda(t)$ and $C(X, L^1_0G)$ its kernel. Choose any positive norm one element $\psi$ in $L^1G$ and consider the map from $C(X)$ to $C(X, L^1G)$ defined by $h \mapsto h \otimes \psi$, where $h \otimes \psi(x) = h(x) \psi$. Observing that this map is a continuous linear section of $I$ and passing to the bidual sequence, we can conclude by Theorem 8.2.7 of~\cite{These_Monod}.

We conclude this introduction with the following warning. In this document, all the topological spaces the reader will come across have the Hausdorff property.
\chapter{Acknowledgments}
First of all, I naturally would like to thank Professor Nicolas Monod, who has supervised my Master thesis. Nicola-san, arigatoo gosaimasu ! Every time I went into his office, I came out less ignorant about much more than I first expected to. Enlightening, honest and non pedantic answers have been the rule during three months. I am thankful for the total freedom he left me in my work.

I also would like to thank two people for their localized but useful (and also enjoyable) help. I am thinking to Professor Clair Anantharaman-Delaroche, who answered to the newbie e-mail I sent her and to Professor Antoine Derighetti, who showed me in a record time how much some problems have to be solved.

Finally, in order to honor my promise, I thank Kathleen Chevalley, who has surely helped me for this work. Somehow, being unable to find a mathematically aware person to talk to at home is not such a terrible thing.

\mainmatter
\pagenumbering{arabic}

\chapter{Banach lattices}\label{lattices:Ch}
  This chapter aims to introduce - in a self contained way - the minimal knowledge about Banach lattices enabling the reader to understand the subsequent developments. The main goal is to establish that the topological dual of a Banach lattice is also a Banach lattice. This is why we will quickly turn to maps between vector lattices and dual spaces. Though it will not be used in the subsequent development, it will be established at the end of the chapter - because it is of general interest - that the canonical embedding of a Banach lattice into its double dual is a lattice homomorphism.

Inspiration was mainly drawn from~\cite{metric} for the straightforward approach and from~\cite{scha} for details.

\section{Definitions and basic properties}

Before introducing Banach lattices, we develop some background theory about general vector lattices.

\begin{defn}
 A lattice $(V, \leq)$ is a \emph{vector lattice}\index{vector lattice} if the underlying set $V$ is a vector space over $\R$ such that addition and multiplication by nonnegative scalars preserve the order. In other words, for every $x, y$ in $V$, the relation $x \leq y$ implies that $x \pm z \leq y \pm z$ and $\lambda x \leq \lambda y$ for every $z$ in $V$ and every $\lambda > 0$. Because of a misuse of language, the order is generally not referred to and $V$ is said to be a vector lattice.
\end{defn}

Let $V$ be a vector lattice. For two elements $x, y$ in $V$, we denote by $x \supr y$\index{max@$\supr$} and $x \infi y$\index{max@$\infi$} the supremum, respectively the infimum, of the set $\{x, y\}$. The supremum and infimum are called the \emph{lattice operations}\index{lattice operations}. We define the positive and negative parts of $x$ by $x^+ := x \supr 0$\index{x plus@$x^+$} and $x^-:=0 \supr -x$\index{x minus@$x^-$}. We also define the \emph{absolute value} of $x$ by $|x|:=x^+ \supr x^-$. Finally, we define the \emph{positive cone} of $V$ by $V^+:=\{x \in V : x \geq 0\}$\index{V plus@$V^+$}.

\begin{prop}\label{prorietes_lattices}
 Let $V$ be a vector lattice. For every $x, x', y, y', z$ in $V$ and every non empty subset $A$ of $V$, the following statements hold.
\begin{itemize}
  \item[(a)]If $x \leq y$, then $-x \geq -y$.
  \item[(b)]For $\lambda >0$, the equalities $- \sup A = \inf (-A)$ and $\sup(\lambda A) = \lambda \sup(A)$ hold, whenever one of their sides exists.
  \item[(c)]The equality $z + \sup A = \sup(z + A)$ holds, whenever one of its sides exists. The infimum analogue is also true.
  \item[(d)]The equality $x+y = (x \supr y) + (x \infi y)$ holds.
  \item[(e)]The equalities $x = x^+ - x^-$, $x^+ \infi x^- = 0$ and $|x| = x^+ + x^-$ hold.
  \item[(f)]We have the following unique decomposition : if $x= u - v$ with $u, v$ in $V^+$ and $u \infi v = 0$, then $u = x^+$ and $v = x^-$.
  \item[(g)]We have $|x|= 0$ if and only if $x = 0$.
  \item[(h)]For every $\lambda$ in $\R$, the equality and inequality $|\lambda x | = |\lambda| |x|$ and $|x + y| \leq |x| + |y|$ hold.
  \item[(i)]The equality $|x - y| = (x \supr y) - (x \infi y)$ holds.
  \item[(j)]We have the following distributive laws : if $\{x_\alpha\}_{\alpha \in B}$ and $\{y_\alpha\}_{\alpha \in B}$ are two subset of $V$ such that $\sup_\alpha x_\alpha $ and $\inf_\alpha y_\alpha $ exist, then $z \infi\sup_\alpha x_\alpha= \sup_\alpha (x_\alpha \infi z)$ and  $z \supr \inf_\alpha y_\alpha = \inf_\alpha (y_\alpha \supr z)$.
  \item[(k)]The equality $|x| = x \supr -x$ holds.
  \item[(l)]The two equalities $|(x \supr y )-(x' \supr y')| \leq |x- x'| + |y - y'|$ and $|(x \infi y )-(x' \infi y')| \leq |x- x'| + |y - y'|$ hold.
\end{itemize}
\end{prop}

This proposition will not be proved. These properties are given so that the reader can follow without any troubles the subsequent discussion. The proofs of these properties can be found in~\cite{scha}.

Let $V$ be a vector lattice. If $x$ and $y$ belong to $V$, we define $[x, y]$\index{interval@$[x, y]$} to be the set of all the elements $z$ in $V$ such that $x \leq z \leq y$. With this notation, we have the following decomposition property.

\begin{prop}[Decomposition property\index{decomposition property}]\label{decomposition_property}
If $x$ and $y$ are two positive elements of a vector lattice, then $[0, x+y] = [0, x] + [0, y]$
\end{prop}
\begin{proof}
 Let $x$ and $y$ be two positive elements of a vector lattice. It is clear that $[0, x] + [0, y] \subseteq [0, x+y]$. Let us prove the other inclusion. Let $z$ belong to $[0, x+y]$. Define $u := x \infi z$ and $v := z-u$. These definitions imply that $u$ belongs to $[0, x]$ and $u + v = z$. The proposition will be proved once shown that $v$ belongs to $[0, y]$. Replacing $u$ in the definition of $v$ by $x \infi z$ leads to $v = (z-x)\supr 0$ (by parts (b) and (c) of Proposition~\ref{prorietes_lattices}) and since $z \leq x + y$, we have $v \leq (x+y-x) \supr 0 = y$.
\end{proof}

\begin{cor}\label{cor_decomposition_property}
 If $x, y$ and $z$ are positive elements of a vector lattice, then $(x+y)\infi z \leq (x\infi z) + (y \infi z)$.
\end{cor}
\begin{proof}
 Let $x, y$ and $z$ be positive elements of a vector lattice and define $v:=(x+y) \infi z$. This definition implies that $v$ belongs to $[0, x+y]$ and by the decomposition property, we can find $v_1$ in $[0, x]$ and $v_2$ in $[0, y]$ such that $v = v_1+v_2$. Since $v_1 \leq v_1 + v_2 = (x+y) \infi z$, we have $v_1 \leq x\infi z$. Similarly, $v_2 \leq y \infi z$. Hence, we find that $v_1 + v_2 \leq (x+z) \infi (y+z)$.
\end{proof}

When the underlying vector space of a vector lattice is a normed space, we may wish to have some kind of compatibility between the order and the norm structures. This gives rise to the definition of normed lattices.

\begin{defn}
 Let $(V, \leq)$ be a vector lattice. A norm on $V$ is a \emph{lattice norm}\index{lattice norm} if $|x| \leq |y|$ implies that $\|x\| \leq \|y\|$ for every $x$ and $y$ in $V$. If $\|\cdot\|$ is a lattice norm on $V$, it is said that $(V, \leq, \|\cdot\|)$ or $(V, \|\cdot\|)$ or even just $V$ is a \emph{normed lattice}\index{normed lattice}. If $(V, \|\cdot\|)$ is a normed lattice and if moreover $V$ is complete with respect to $\|\cdot\|$, then $V$ is a \emph{Banach lattice}\index{Banach lattice}.
\end{defn}

The following proposition gives an alternative definition of normed lattices. It also says that in normed lattices, taking absolute values preserves the norm.

\begin{prop}\label{prop1.01.12}
 If $V$ is a normed lattice, then for every $x$ and $y$ in $V$, we have $\| x \| = \|\,|x|\,\|$ and if $0 \leq x \leq y$, then $\|x\| \leq \|y\|$. Conversely, if $V$ is a vector lattice endowed with a norm that verifies both of the previous conditions, then $V$ is a normed lattice.
\end{prop}
\begin{proof}
 Assume that $V$ is a normed lattice. Since taking twice the absolute value is the same as taking it once, the very definition of a lattice norm implies that $\|x\| = \| \,|x|\,\|$ for every $x$ in $V$. The second fact stated above follows from the fact that the absolute value of any positive element is the element itself.

Suppose now that $V$ is a vector lattice and that $\|\cdot\|$ is a norm on $V$ satisfying both of the conditions in the statement of the proposition. Let $x$ and $y$ in $V$ be such that $|x| \leq |y|$. Because $0 \leq |x|, |y|$, our assumptions imply that $\| x \| = \| \, |x| \, \| \leq \| \, |y| \, \| = \| y\|$.
\end{proof}

\begin{prop}\label{properties_normes_lattices}
 Let $V$ be a normed lattice. The following properties hold.
\begin{itemize}
  \item[(a)]The lattice operations $\supr$ and $\infi$ are uniformly continuous.
  \item[(b)]The positive cone $V^+$ is closed for the norm topology.
  \item[(c)]If $\{x_n\}_{n \in \N}$ is an increasing and converging sequence of $V$, then its supremum exists and is equal to its limit.
\end{itemize}
\end{prop}
\begin{proof}
 We only prove part (a) for $\supr$, the proof being similar for $\infi$. Assume that $\{x_n\}_{n \in \N}$ and $\{y_n\}_{n \in \N}$ are two sequences in $V$ converging to $x$ and $y$, respectively. By part (l) of Proposition~\ref{prorietes_lattices}, we have
$$
|(x_n \supr y_n) - (x \supr y)| \leq |x_n - x|+|y_n - y|
$$
and Proposition~\ref{prop1.01.12} applied to this inequality leads to
$$
\|(x_n \supr y_n) - (x \supr y)\| \leq \|x_n - x\|+\|y_n - y\|.
$$
This proves the uniform continuity of the supremum. Part (b) is proved by means of part (a). In fact, if $\{x_n\}_{n \in \N}$ is a sequence in $V^+$ that converges to some $x$ in $V$, then
$$
x = \lim_{n \to \infty} x_n = \lim_{n \to \infty} (x_n \supr 0) = x\supr 0
$$
because the supremum is continuous. This equality means that $x$ belongs to $V^+$.

Finally, assume that $\{x_n\}_{n \in \N}$ is an increasing and converging sequence in $V$ and that $x$ is its limit. For every integers $m$ and $n$ with $m \geq n$ the inequality $x_m - x_n \geq 0$ holds. Letting $m$ tend to infinity, we obtain $x - x_n \geq 0$ for every $n$ and therefore $x$ is an upper bound for $\{x_n\}_{n \in \N}$. If $y$ is an other upper bound for the $x_n$'s, then $u-x$ belongs to $V^+$ because every $u-x_n$ is in $V^+$ and this subset is closed by part (b). Hence $x \leq u$ and $x$ is indeed the supremum of the $x_n$'s.
\end{proof}
%
%
\section{Maps between vector lattices}

In this section two results about maps between vector lattices will be presented. The first one shows how to define linear maps between vector lattices. Lattice homomorphisms will be defined below. The second result gives an alternative definition of these homomorphisms.

\begin{defn}
Let $T : V \rightarrow W$ be a linear map between vector lattices. The map $T$ is \emph{positive}\index{positive map} if it maps the positive cone $V^+$ into the positive cone $W^+$. We say that $T$ is a \emph{lattice homomorphism}\index{lattice homomorphism} if it preserves the lattice structure. In other words, $T$ is a lattice homomorphism if $T(x\supr y) = Tx \supr Ty$ for every $x$ and $y$ in $V$.
\end{defn}

Since every element of a vector lattice is the difference of its positive and negative parts, the values of a linear map are entirely determined by its values on positive elements.

\begin{prop}\label{def_on_cones}
 Let $t : V^+ \rightarrow W$ be a map from the positive cone of a vector lattice $V$ into a vector lattice $W$. If $t$ is additive and homogeneous over $\R_+$, then there is a unique linear extension $T : V \rightarrow W$ of $t$. Moreover, if $t$ is positive, then so is $T$. 
\end{prop}
\begin{proof}
 Let $t$ be as in the statement of the proposition and recall that by part (e) of Proposition~\ref{prorietes_lattices}, we have $V = V^+-V^+$. Define $T : V \rightarrow W$ by declaring that if $x = u-v$ with $u$ and $v$ in $V^+$, then $T(x)=t(u)-t(v)$. Because of the linearity and the positive homogeneity of $t$, the map $T$ is well defined and is linear. Since every linear extension of $t$ has to verify the defining equation of $T$, this extension is unique.

Because $T$ is an extension of $t$, it coincides with $t$ on the positive cone of $V$. Therefore if $t$ is positive, then so is $T$.
\end{proof}

Lattice homomorphisms have been defined by their compatibility the supremum operation. It is obvious that we could equivalently have defined them with the infimum. The following proposition says that we could also have defined them with the positive part operation.

\begin{prop}\label{homo_positive_part}
 A linear map $T : V \rightarrow W$ between vector lattices is a lattice homomorphism if and only if $T(x^+)=(Tx)^+$ for every $x$ in $V$.
\end{prop}
\begin{proof}
 Let $T : V \rightarrow W$ be a linear map between vector lattices. If $T$ is a lattice homomorphism, the definitions immediately imply that $T(x^+)=(Tx)^+$ for every $x$ in $V$. Conversely, assume that $T(x^+)=(Tx)^+$ for every $x$ in $V$. An application of part (c) of Proposition~\ref{prorietes_lattices} shows that $x \supr y = y + (x-y)^+$ for every $x$ and $y$ in $V$. This identity, the linearity of $T$ and our assumption that $T$ commutes with the positive part operation imply that $T(x \supr y) = Tx \supr Ty$ for every $x$ and $y$ in $V$.
\end{proof}
%
%
\section{Dual spaces}

If $V$ is a vector lattice, then its algebraic dual $V'$ carries a natural order structure.

\begin{defn}
 Let $V$ be a vector lattice. The \emph{canonical order}\index{canonical order} on its algebraic dual $V'$ is defined by declaring that $f \leq g$ if and only if $f(x)\leq g(x)$ for every $x$ in $V^+$.
\end{defn}

We wish to show that if $V$ is a Banach lattice, then the canonical order on $V'$ induces a vector lattice structure on its topological dual $V^*$. Because it is not easy to prove that the lattice operations preserve the topological dual, an auxiliary subspace $V^\#$ of $V'$ will first be defined. The space $V^\#$ will then be shown to be a vector lattice and finally it is shown that $V^*$ and $V^\#$ coincide.

\begin{defn}
 Let $V$ be a vector lattice. A linear functional $f : V \rightarrow \R$ is \emph{order bounded}\index{order bounded} if $f([x, y])$ is a bounded subset of $\R$ for every $x$ and $y$ in $V$. We define $V^\#$\index{sharp@$^\#$} to be the subspace of the algebraic dual of $V$ containing all of the order bounded functionals. In other words, $V^\# :=\{ f \in V' : f \text{ is order bounded}\}$.
\end{defn}

\begin{prop}\label{V_sharp_vector_lattice}
 Let $V$ be a vector lattice and $V'$ its algebraic dual. Endowed with the canonical order of $V'$, the subspace $V^\#$ is a vector lattice and its lattice operations are given on $V^+$ by
\begin{eqnarray}
(f\supr g )(x) &=& \sup \{f(y) + g(x-y) : y \in [0, x]\} \text{ and } \label{eq1:02.12}\\
(f\infi g )(x) &=& \inf \{f(y) + g(x-y) : y \in [0, x]\}.\label{eq2:02.12}
\end{eqnarray}
\end{prop}
\begin{proof}
 Only equation~\eqref{eq1:02.12} will be proved. To that end, we will define a map $h$ on $V^+$ that verifies all the properties that $f \supr g$ should verify. Then, we show that $h$ can be extended to $V$ and that it belongs to $V^ \#$. Define $h$ by declaring that
$$
h(x) := \sup \{f(y) + g(x-y) : 0 \leq y \leq x \} \quad (x \in V^+).
$$
If $k$ is an upper bound of $\{f, g\}$ in $V^\#$, then for every $x$ in $V^+$ we have $h(x)\leq k(x)$ and $f(x), g(x) \leq h(x)$ for every $x$ in $V^+$. Thus, if we show that $h$ induces a linear order bounded functional, then we are done.

First, we show that $h$ can be extended into a linear functional. According to Proposition~\ref{def_on_cones}, it suffices to show that $h$ is additive and positively homogeneous. The latter property is no trouble and we immediately tackle the additivity issue. Let $x_1$ and $x_2$ belong to $V^+$ and set $x:= x_1 + x_2$. By the Decomposition property~\ref{decomposition_property}, we have $[0, x] = [0, x_1]+[0, x_2]$ and so
$$
h(x)=\sup\{f(y_1+y_2) + g(x-(y_1 + y_2)) : y_1 \in [0, x_1], y_2 \in [0, x_2]\}.
$$
By writing $x$ as $x_1 + x_2$ in the above equation, using the linearity of $f$ and $g$ and splitting the supremum in two, we obtain $h(x)=h(x_1) + h(x_2)$.

We end this proof by showing that the unique linear extension of $h$ is order bounded. By translating intervals, it is enough to show that $h([0, x])$ is bounded for every $x$ in $V^+$. Notice that for every $z$ in $[0, x]$ we have
$$
f(z) \leq h(z) \leq \sup_{0 \leq y \leq x} f(y) + \sup_{0 \leq y \leq x} g(y).
$$
The functionals $f$ and $g$ being order bounded, this inequality implies that $h([0, x])$ is bounded.
\end{proof}

\begin{cor}\label{absolute_value_dual}
 The positive part and the absolute value of an order bounded functional $f$ are given on $V^+$ by the formulas
\begin{eqnarray*}
f^+(x) &=& \sup \{f(v) : v \in [0, x]\} \text{ and }\\
|f|(x)&=& \sup\{f(v) : |v| \leq x\}.
\end{eqnarray*}
\end{cor}
\begin{proof}
The formula for $f^+$ is obtained by letting $g=0$ in the formula for the supremum in Proposition~\ref{V_sharp_vector_lattice}. By part (k) of Proposition~\ref{prorietes_lattices}, we have $|f| = f \supr -f$. The formula for the supremum leads to
$$
|f|(x)=\{f(2y -x) : y \in [0, x]\}.
$$
The observation that $\{2y - x : y \in [0, x]\} = \{v : |v| \leq x\}$ linked to the equation above ends the proof.
\end{proof}

We are one step away from establishing that in the complete case, $V^*$ is nothing else than $V^\#$.

\begin{prop}\label{positive_continuous}
Let $V$ and $W$ be normed lattices. If $V$ is complete, then every positive linear map $T : V \rightarrow W$ is continuous. 
\end{prop}
\begin{proof}
 Let $V$ and $W$ be normed lattices and assume that $V$ is complete. Let $B$ denote the closed unit ball in $V$ and observe that if $B^+ := B \cap V^+$, then $B$ is contained in $B^+ - B^+$. This is true because the norms of the positive and negative parts of any vector are always less or equal to the norm of its absolute value.

Let $T : V \rightarrow W$ be a positive linear map and suppose for the sake of contradiction that $T$ is not continuous. In particular, $T$ is not bounded on $B$ and therefore neither is it on $B^+$. Thus our assumption that $T$ is not continuous produces a sequence $\{x_n\}_{n \in \N}$ in $B^+$ such that $\|Tx_n\| \geq n^3$ for every integer $n$. Because $V$ is complete, a vector $z$ can be defined by $z:=\sum_{i = 1}^\infty x_i/i^2$. The closedness of positive cones in normed lattices implies that $z$ belongs to $V^+$. Because $z \geq x_n/n^2 \geq 0$ for every $n$ and the norm being monotonically increasing on the positive cone, we obtain the contradiction that $\|Tz\| \geq n$ for every $n$.
\end{proof}

\begin{thm}\label{V_sharp_V_star}
 If $V$ is a Banach lattice, then $V^*=V^\#$ and $V^*$ is a Banach lattice whose lattice operations are given on $V^+$ by~\eqref{eq1:02.12} and~\eqref{eq2:02.12}. 
\end{thm}
\begin{proof}
 Let $f$ belong to $V^*$. Using part (k) of Proposition~\ref{prorietes_lattices}, it is easy to see that for every $x, y$ and $z$ in $V$ such that $z$ belongs to $[x,y]$, the inequality $0 \leq |z| \leq |x|+|y|$ holds. This implies that $|fz| \leq \|f\| (\|x\|+\|y\|)$ for every $z$ in $[x, y]$ and so $f$ is order bounded.

Conversely, assume that $f$ belongs to $V^\#$. Since $V^\#$ is a vector lattice, we have the decomposition $f= f^+ - f^-$. The definitions are such that every functional belonging to the positive cone of $V'$ is positive. But then Proposition~\ref{positive_continuous} asserts that $f$ is the difference of two continuous maps. Thus $f$ is continuous.

We have just seen that $V^\# = V^*$, which shows that $V^*$ is a vector lattice for the lattice operations mentioned in the statement. It must still be proved that it is a Banach lattice. It can be achieved by means of Proposition~\ref{prop1.01.12}. Before rushing into the proof, write a functional $f$ and a vector $x$ as $f^+ - f^-$ and $x^+ - x^-$  and use the triangle inequality to obtain the estimate
\begin{equation}\label{eq1:03.12}
|f(x)| \leq |f^+(x^+)|+ |f^+(x^-)|+|f^-(x^+)|+|f^-(x^-)| = |f|(|x|).
\end{equation}
Now, let us verify that $\| \,|f|\,\| = \|f\|$ for every $f$ in $V^*$. By the inequality above, we have
\begin{eqnarray*}
\sup \{|f(x)|: \|x\| \leq 1\} 	& \leq &	\sup \{|f|(|x|) : \|x\| \leq 1\}  \\
				& = & 		\sup \{|f|(x) : x \geq 0 \text{ and } \|x\| \leq 1\}\\
				& = &		\sup \{|\,|f|(x)\,| : x \geq 0 \text{ and } \|x\| \leq 1\} \leq \| \,|f|\, \|.
\end{eqnarray*}
This shows that $\|f\| \leq \| \,|f|\, \|$. Conversely, using~\eqref{eq1:03.12} on $|f|$, we have
\begin{eqnarray*}
\sup \{|\,|f|(x)|: \|x\| \leq 1\}		& \leq &	\sup \{|f|(|x|) : \|x\| \leq 1\}  \\
				  & = & 		\sup \{|f|(x) : x \geq 0 \text{ and } \|x\| \leq 1\}\\
				  & = &		\sup \{ \sup\{f(v) : |v| \leq x\} :  x \geq 0 \text{ and } \|x\| \leq 1 \}\\
				  &\leq&	\sup \{\|f\| :  x \geq 0 \text{ and } \|x\| \leq 1\}.
\end{eqnarray*}
The last term being equal to $\|f\|$, we have $\| \,|f|\,\| \leq \|f\|$.

It remains to be shown that $\|f\| \leq \|g\|$ whenever $0 \leq f \leq g$. Again, use equation~\eqref{eq1:03.12} to obtain
\begin{eqnarray*}
 \sup \{|f(x)|: \|x\| \leq 1\} 	& \leq &  \sup \{|f|(|x|) : \|x\| \leq 1\}  \\
				& \leq & \sup \{|g|(|x|) : \|x\| \leq 1\} \\
			      &=& \sup \{|g|(x) : x \geq 0 \text{ and } \|x\| \leq 1\}.\\
\end{eqnarray*}
Since the last term is less or equal to $\|g\|$, we are done.
\end{proof}

We finally tackle the double dual topic. We will establish two results. The first one says that the canonical embedding of a Banach lattice into its double dual is a positive map and the second one is the result announced in the opening of the chapter : the canonical embedding is a lattice homomorphism.

\begin{thm}\label{embedding_positive}
 Let $V$ be a Banach lattice. If $J$ denotes the canonical embedding of $V$ into its double dual $V^{**}$, then for every $x$ in $V$ we have $x \geq 0$ if and only if $Jx \geq 0$. In particular, $J$ is positive.
\end{thm}
\begin{proof}
Assume $x$ belongs to $V^+$ and let $f$ be a bounded linear functional satisfing $f \geq 0$. That $Jx(f) \geq 0$ is obvious, since $Jx(f) = f(x)$ and both $x$ and $f$ are greater or equal to $0$. Thus, $Jx \geq 0$.

Conversely, assume that $Jx \geq 0$ and let us show that $x \geq 0$. Suppose $x$ does not belong to the closed subset $V^+$. Then there would exist an $f$ in $V^*$ such that $f(x)<0$ and $f(y) \geq 0$ for every $y$ in $V^+$. In particular, $f$ is positive and therefore $Jx(f) \geq 0$. This is absurd because $Jx(f) = f(x) <0$ and therefore $x$ has to belong to $V^+$.
\end{proof}

\begin{thm}\label{embedding_homomorphism}
 The canonical embedding of a Banach lattice into its double dual is a lattice homomorphism.
\end{thm}

For the sake of clarity, the proof of the above theorem is prepared with a lemma.

\begin{lem}\label{lem1:03.12}
 Let $V$ be a Banach lattice and $f$ a positive element of $V^*$. If $x$ is any element of $V$, then the map $\xi$ defined on $V^+$ by
$$
\xi(y) := \sup\{\ f(z) : z \in [0, y] \cap E\},
$$
where $E:= \cup_{n \in \N} n[0, x^+]$ verifies the following properties.
\begin{itemize}
 \item[(a)]There is an unique positive bounded linear functional on $V$ extending $\xi$.
  \item[(b)]If we also denote this extension by $\xi$, then $\xi \leq f$ in $V^*$.
  \item[(c)]The equality $\xi(x^-) = 0$ holds.
\end{itemize}
\end{lem}
\begin{proof}
 Let $x$ belong to $V$ and define $\xi$ on $V^+$ as in the statement of the lemma. It is easy to show that $\xi$ is positively homogeneous. Using Decomposition property~\ref{decomposition_property} one also shows that it is additive. By Proposition~\ref{def_on_cones}, $\xi$ admits an unique positive extension in the algebraic dual of $V$. By Proposition~\ref{positive_continuous}, this extension belongs to $V^*$. This proves part (a). Since $f$ is positive, part (b) is true because
\begin{eqnarray*}
\xi(y) &=& \sup  \{f(z) : z \in [0, y] \text{ and } z \in E\}\\
      &\leq& \sup  \{f(z) : z \in [0, y] \} \leq f(y).
\end{eqnarray*}

To prove part (c), pick $\epsilon >0$ and a vector $z$ in $[0, x^-]\cap E$ such that $\xi(x^-) \leq f(z) + \epsilon$. Two cases can be distinguished. If the only integer $n$ such that $z \in n[0, x^+]$ is $0$, then $z = 0$ and it follows that $ 0 \leq \xi (x^-) \leq \epsilon$. If not, then we can find an integer $n > 0$ such that
$$
z \leq x^- \infi nx^+ \leq n(x^- \infi x^+) = 0.
$$
Thus, in both cases, $0 \leq \xi (x^-) \leq \epsilon$. Since $\epsilon$ is arbitrary, it is true that $\xi(x^-) = 0$.
\end{proof}

\begin{proof}[Proof of Theorem~\ref{embedding_homomorphism}]
 Let $V$ be a Banach lattice and let $J$ denote the canonical embedding of $V$ into its double dual. We will show that $(Jx)^+ = Jx^+$ for every $x$ in $V$ and conclude using Proposition~\ref{homo_positive_part}.

By Theorem~\ref{embedding_positive}, the map $J$ is positive. This implies that $J$ is monotonically increasing and using these two properties, we have
$$
(Jx)^+ = Jx \supr 0 \leq Jx^+ \supr 0 = Jx^+
$$
for every $x$ in $V$. This is a half of the result.

To show the reverse inequality, let $f$ be any positive element of $V^*$ and $x$ belong to $V$. For these $f$ and $x$, define $\xi$ as in Lemma~\ref{lem1:03.12} and write $\xi$ its unique extension in $V^*$. Since $\xi (x^-) = 0$, we have $\xi(x) = f(x^+)$. Using Corollary~\ref{absolute_value_dual} to justify the last equality below, we see that
$$
Jx^+ (f) = f(x^+) = \xi(x) \leq \sup \{g(x) : g \in [0, f]\} = (Jx)^+(f).
$$
This shows that $Jx^+  \leq (Jx)^+$.
\end{proof}
%
%
\section{Basic constructions}

We describe now some basic constructions around vector lattices that will be used later in a particular case. Starting with any set and a vector lattice, new vector lattices will be constructed. Depending on the hypothesis on these spaces, we will see what kind of vector lattice we obtain. We will proceed in the form of a discussion and resume the results at the end of the section in a proposition. We recall that by default, all the following topological spaces have the Hausdorff property.

Let $X$ be a set and $V$ a vector lattice. We define a partial order on $V^X$ by declaring that $f \leq g$ if and only if $f(x) \leq g(x)$ for every $x$ in $X$. We call this order the \emph{pointwise order}\index{pointwise order}. Under that order, $V^X$ is a vector lattice whose lattice operations are given by $(f\supr g)(x) = f(x) \supr g(x)$. In particular, the absolute value is given by $|f|(x) = |f(x)|$.

Assume now that $X$ is a topological space and that $V$ is a topological vector space. The pointwise order on $V^X$ induces an order on $C(X,V)$. It is almost obvious that it is a lattice order whenever the map $\supr : V \times V \rightarrow V$ is continuous. In particular, if $V$ is a normed lattice, then $C(X, V)$ is a vector lattice.

Assume moreover that $X$ is compact and that $V$ is a normed lattice. We endow $C(X, V)$ with the uniform norm. If $f$ and $g$ are two maps in $C(X,V)$, then $|f| \leq |g|$ if and only if $|f(x)| \leq |g(x)|$ for every $x$ in $X$. The norm on $V$ being compatible with the order, this implies that $\|f(x)\| \leq \|g(x)\|$ for every $x$ in $X$ and so $\|f\| \leq \|g\|$. This shows that under our assumptions $C(X, V)$ is a normed lattice.

Finally, assume moreover that $V$ is complete. In this case, $C(X, V)$ is also a complete space and therefore $C(X, V)$ is a Banach lattice.

\begin{prop}\label{constructions_lattices}
 Let $X$ be a set and $V$ a vector lattice.
\begin{itemize}
  \item[(a)]The set $V^X$ is a vector lattice under the pointwise order.
  \item[(b)]If $X$ is a topological space and $V$ a topological vector space, then $C(X, V)$ is a vector lattice whenever the supremum operation is continuous on $V \times V$.
  \item[(c)]If $X$ is a compact space and $V$ a normed lattice, then $C(X, V)$ endowed with the uniform norm is a normed lattice. Moreover, if $V$ is complete, then $C(X, V)$ is a Banach lattice.
\end{itemize}
\end{prop}

\begin{exple}\label{C(X, L1)_lattice:exple}
 Let $G$ be a locally compact topological group endowed with its left Haar measure. We claim that if $X$ is a compact space, then the set $C(X, L^1G)$ endowed with the uniform norm is a Banach lattice for some order. By part (c) of Proposition~\ref{constructions_lattices}, it is enough to show that $L^1G$ is a normed lattice for some order.

It is not hard to see that the pointwise order on $\R^G$ induces a well defined lattice order on the equivalence classes of $L^1(G)$ : the order is defined by declaring that $[f] \leq [g]$ if and only if $f \leq g$ almost everywhere, with $[f], [g]$ in $L^1G$. For every $[f]$ in $L^1G$, one can see that $[f]^\pm =  [f^\pm]$ and therefore $|[f]| = [|f|]$.

It immediately follows from these observations that $(L^1G, \| \cdot \|_1)$ is indeed a normed lattice. Since this normed space is complete, Proposition~\ref{constructions_lattices} applies to $C(X, L^1G)$.
\end{exple}

\chapter{The Bochner integral}\label{integral:Ch}
  The goal of this chapter is to introduce the \emph{Bochner integral}, which will enable us to integrate functions defined on some measure space and valued in a Banach space. It will be defined from the very beginning and some basic (but useful) results concerning the integral will be proved.

We owe this simple and concise presentation of the subject to a German note found on the Internet (see~\cite{bochn}).  Some generality will be added to this online presentation by introducing some weaker ``almost everywhere'' hypothesis.

Again, we recall that all the following topological spaces have the Hausdorff property.

\section{Functions spaces}\label{section_functions_spaces}

Even if most of the following definitions and results make sense and hold in a
more general setting than ours, the emphasis is put on the Banach space case
since we ultimately seek to define the Bochner integral for Banach space valued maps
only. When it will come to measurability questions, the Banach spaces will
always be endowed with their Borel $\sigma$-field.

\begin{defn}
Let $s, f, g : \Om \rightarrow B$ be maps from a measure space $(\Om,\mathcal{F}, \mu)$ to a Banach space $B$. The map $s$ is \emph{simple}\index{simple map} if
it is measurable and $s(\Om)$ is a finite set. In other words, $s$ is simple if
there are finitly many disjoint  sets $A_1, \dots, A_n$ in $\mathcal{F}$ such
that
$$
s = \sum_{i=1}^n \indicc{A_i} b_i,
$$
where the $b_i$'s belong to $B$. The map $f$ is \emph{almost everywhere
separable}\index{separable map} if there exist some $\mu$-null set $N$ in $\mathcal{F}$ and some
countable subset $Y$ of $B$ such that $f(\Om \backslash N)$ is contained in
$\overline{Y}$. Because of a misuse of language, such a map is said to be
\emph{separable}. Finally, the map $g$ is $\mu$-\emph{measurable}\index{mu-measurable@$\mu$-measurable map} if there is a
sequence $\{s_n\}_{n \in \N}$ of simple maps that converges pointwise to $g$
almost everywhere on $\Om$.
\end{defn}

Our first few steps will be to obtain some results about the classes of measurable
and separable maps.

\begin{lem}\label{lemma1:29.11}
 Let $\{f_n\}_{n \in \N}$ and $f$ be maps from a measure space $\Omega$ to a Banach space $B$ such that the $f_n$'s converge pointwise to $f$ almost everywhere on $\Om$.
\begin{itemize}
	\item[(a)] If all of the $f_n$'s are separable, then so is $f$.
	\item[(b)] If all of the $f_n$'s are measurable, then so is $f$.
\end{itemize}
\end{lem}

\begin{proof}
	To prove part (a), let $M$ be a null subset of $\Om$ on the complement of which the $f_n$'s converge pointwise to $f$. For every interger $n$, let
$N_n$ be a null subset of $\Om$ and $Y_n$ a countable subset of $B$
such that $f_n(\Om \backslash N_n)$ is contained in $\overline{Y_n}$. Set
$$
N:= M \cup \left( \bigcup_{n \in \N} N_n \right)	 \quad \mbox{ and }
\quad Y := \bigcup_{n \in \N} Y_n.
$$
The subset $N$ is of measure zero and $Y$ is countable. Moreover, $f(\Om
\backslash N)$ is contained in $\overline{Y}$ and therefore $f$ is separable.

Let us prove part (b). Since the Borel $\sigma$-field of $B$ is generated by the closed subsets of
$B$, it is enough to see that the pre-images under $f$ of closed subsets of $B$ are measurable subsets of $\Om$. For a closed subset $A$ of $B$ and every $x$ in $B$, let
$g(x)$ denote the distance from $x$ to $A$. The function $g$ is continuous and
$A = g^{-1}(\{0\})$. The continuity of $g$ implies that the $(g \circ f_n)$'s converge pointwise to
$g \circ f$ almost everywhere on $\Om$. Since the almost everywhere
pointwise limit of an $\R$-valued sequence of measurable maps is measurable, $g
\circ f$ is measurable. Thus $f^{-1}(A)$, which is equal to $(g \circ f)^{-1}(\{0\})$, is
measurable.
\end{proof}

\begin{lem}\label{lemme2:29.11}
	Let $f, g : \Om \rightarrow B$ be maps from a measure space
$\Om$ to a Banach space $B$. If $f$ and $g$ are separable, then so is $f+g$.
\end{lem}
\begin{proof}
	Assume that $f$ and $g$ are separable. Let $N_1, N_2$
be null subsets of $\Om$ and $Y_1, Y_2$ countable subsets of $B$ witnessing the separability of $f$ and $g$, respectively. If $N$ is the union of $N_1$ and $N_2$, then $(f+g)(\Om \backslash N)$ is contained in $\overline{Y_1 + Y_2}$. Since $N$ is of measure zero and $Y_1 + Y_2$ is countable, the map $f+g$ is separable.
\end{proof}

The following lemma will help us to translate some statements about separable maps
into statements about $\mu$-measurable maps and conversely.

\begin{lem}\label{equivalence1}
	Let $f : \Om \rightarrow B$ be a map from a measure space $\Om$ to a
Banach space $B$. The following statements are equivalent.
\begin{itemize}
	\item[(a)]The map $f$ is $\mu$-measurable and the sequence
$\{s_n\}_{n \in \N}$ in the definition of $\mu$-measurability can be choosen in
such a way that $\|s_n (\om)\| \leq 2\| f(\om)\|$ almost everywhere.
	\item[(b)]The map $f$ is measurable and separable.
\end{itemize}
\end{lem}

\begin{rmq}
 It must be stressed that the domination condition in part (a) of Lemma~\ref{equivalence1} is not used in order to obtain part (b).
\end{rmq}

\begin{proof}
	The fact that (a) implies (b) is an easy application of
Lemma~\ref{lemma1:29.11}. Our task is to prove that (b) implies (a). Let $N$ be a null subset of $\Om$ and $Y$ a countable subset of $B$ witnessing the separability of $f$.
Without loss of generality, we can assume that $Y$ does not contain $0$. In
fact, up to remove $0$ from $Y$ and to add a sequence of non-zero vectors
converging to $0$, we obtain a new countable subset of $B$ whose closure contains
$\overline{Y}$.

Let $\{y_1, y_2, \dots\}$ be an enumeration of $Y$. For $n$ in $\N^*$ and
$\delta >0$, define 
$$
A_n^\delta := f^{-1}\left(B_\delta (0)^c \cap B_\delta(y_n)\right) \quad \mbox{
and } \quad C_n^\delta := \left(A_n^\delta \backslash \bigcup_{1 \leq i < n}
A_i^\delta \right) \cap N^c,
$$
where $B_\delta(y)$ is the open ball around $y$ of radius $\delta$. Since $f$
is measurable, the $C_n^\delta$'s are measurable subsets of $\Om$. For $1 \leq m < n$, define
$$
t_n^m := \sum_{1 \leq i \leq n} \indicc{C_i^{1/m}} y_i.
$$
The time has come to define the sequence we are looking for. Define
$$
s_n(\om):=
\begin{cases}
	0 & \text{ if } t_n^m
(\om) = 0 \text{ for every } 1\leq m \leq n, \\
	t_n^{m_0}(\om) & \text{ otherwise },
\end{cases}
$$
where $m_0 = m_0(\om, n)$ is the greatest integer $m$ in $\{1, 2, \dots, n\}$ such
that $t_n^m (\om) \neq 0$.
A basic rewriting shows that the $s_n$'s are a simple maps. The definitions and
the $m_0$ notation imply that
\begin{equation}\label{eq1:29.11}
\|s_N(\om) - f(\om)\| \leq \frac{1}{m_0(\om, N)} \quad  \left(\omega \in \bigcup_{m \leq N} \bigcup_{i\leq N} C_i^{1/m}\right).
\end{equation}

To show that the $s_n$'s converge pointwise to $f$ almost everywhere, two cases can be distinguished. First, assume that $\om$ in $N^c$ is such that $f(\om) = 0$. This
means that $\om$ does not belong to any of the $C^\delta_n$'s defined above and
therefore $s_n(\om) = 0$ for every integer $n$. Now assume that $\om$ in  $N^c$
is such that $f(\om) \neq 0$ and let $N_0$ be an integer such that
$$
\|f(\om)\| \geq \frac{1}{N_0}.
$$
Since for a fixed $\delta$ the $C^\delta_n$'s are disjoint, there is a unique $n_0$ such that $\om$ belongs to $C^{1/N_0}_{n_0}$. This implies that for every integer $N$ greater than $N_0$ and
$n_0$
$$
\om \in \bigcup_{1\leq i \leq N} C^{1/N_0}_i
$$
and $m_0 (\om, N) \geq N_0$. Using equation~\eqref{eq1:29.11}, this in turn
implies that for every integer $N$ greater than $N_0$ and $n_0$
$$
\|s_N(\om) - f(\om)\| \leq \frac{1}{m_0} \leq \ \frac{1}{N_0}.
$$
Letting $N_0$ tend to infinity, this shows that the $s_n(\om)$'s  converge to $f(\om)$
for every $\om$ in $N^c$ such that $f(\om) \neq 0$. In conclusion, the $s_n$'s
converge to $f$ pointwise on $N^c$.

We finish the proof by establishing the estimate about $\|s_n (\om)\|$. For
every $\om$ in $C^{1/m}_n$ with $m, n \leq N$,
$$
\|t^m_N(\om)\| = \|y_n\| \leq \|y_n - f(\om)\| + \| f(\om)\| \leq 2 \|f(\om)\|,
$$
by definition of the $A_n^\delta$'s. The very definition of $s_N$ implies that this estimate holds with $s_N$ instead of $t^m_N$ and so
$$
\|s_N (\om) \|\leq 2 \|f(\om)\| \quad \left(\om  \in \bigcup_{m \leq N} \bigcup_{i\leq N} C_i^{1/m}\right).
$$
Since for the other $\om$ in  $N^c$ we have $s_N (\om) = 0$, the estimate holds whenever $\om$ belongs to $N^c$.
\end{proof}

This proof was a little bit tedious but now we have at our disposal a useful equivalence. On one hand, we have the notion of separable function, wich is somewhat abstract but leads easily to results like Lemma~\ref{lemme2:29.11}. On the other hand, the notion of $\mu$-measurable function is very constructive : it is not hard to guess how those functions could be integrated. From now on, the strategy is the following : use the abstract side to elegantly obtain results about $\mu$-measurable functions and the concrete one to define the Bochner integral of $\mu$-mesurable maps.

\begin{lem}\label{lemme1:30.11}
 Let $\Om$ be a measure space and $B$ a Banach space. Define $\mathcal{L}^0 :=\{f : \Om \rightarrow B : f \text{ is measurable and separable}\}$.
\begin{itemize}
 \item[(a)] The subset $\mathcal{L}^0$ is a vector subspace of $B^\Om$.
  \item[(b)] If $s : \Om \rightarrow \R$ is measurable, then $s \mathcal{L}^0 \subseteq \mathcal{L}^0$.
\end{itemize}
\end{lem}
\begin{proof}
 Let $f$ and $g$ belong to $\mathcal{L}^0$. Part (a) is a consequence of Lemma~\ref{lemme2:29.11} and Lemma~\ref{equivalence1}. In fact, the former ensures us that $f+g$ is still separable. The latter ensures us that we can find sequences $f_n$ and $g_n$ of simple maps that converge pointwise to $f$ and $g$ almost everywhere. In particular, the sequence $f_n+g_n$ converges pointwise to $f+g$ almost everywhere. Using Lemma~\ref{lemma1:29.11}, we obtain that $f+g$ is measurable. A similar argument shows that $\R \mathcal{L}^0$ is contained in $\mathcal{L}^0$.

Part (b) is proved by observing that if $s : \Om \rightarrow \R$ is measurable, then $s$ is a measurable and separable function because $\Q$ is dense in $\R$. By Lemma~\ref{equivalence1}, we see that the product $s f$ is almost everywhere the pointwise limit of a sequence of simple maps, whenever $f$ belongs to $\mathcal{L}^0$. Thus $s f$ belongs to $\mathcal{L}^0$.
\end{proof}

\section{The integral}

Let us begin with the integral of simple maps. Let $(\Om, \mathcal{F}, \mu)$ be a measure space and $B$ a Banach space. A simple map $s : \Om \rightarrow B$ with
$$
s = \sum_{i=1}^n \indicc{A_i} y_i
$$
is \emph{integrable} if $\mu(A_i)$ is finite for every $i$ in $\{1, \dots, n\}$. The Bochner integral over $\Om$ of such an integrable simple map is defined by
\begin{equation*}\label{integrale_simple}
\int_\Om s \, d\mu := \sum_{i=1}^n \mu(A_i) y_i.
\end{equation*}
For every simple and integrable map $s$, we have
\begin{equation*}\label{ineg_integrale_simple}
\left\| \int_\Om s \, d\mu \right\| \leq \sum_{i=1}^n \mu(A_i) \|y_i\| = \int_\Om \| s \| \, d\mu.
\end{equation*}

The following theorem defines a sufficiently large class of integrable functions.

\begin{thm}\label{integrables}
 Let $f : \Om \rightarrow B$ be a measurable and separable map from a measure space $(\Om, \mathcal{F}, \mu)$ to a Banach space $B$.
\begin{itemize}
 \item[(a)]There exists a sequence of integrable simple maps $s_n : \Om \rightarrow B$ such that $\int_\Om \|s_n - f\| \, d\mu \rightarrow 0$ if and only if $\int_\Om \|f\| \, d\mu$ is finite.
  \item[(b)]For every sequence like in part (a), the sequence of the $\int_\Om s_n \, d\mu$'s converges in $B$.
  \item[(c)]The limit in part (b) is the same for every two sequences like in part (a).
\end{itemize}
\end{thm}

\begin{proof}
 Let us prove (a). If there is a sequence $s_n$ like in the statement, then
$$
\int_\Om \|f\| \, d\mu \leq \int_\Om \|s_n - f\| \, d\mu + \int_\Om \| s_n \| \, d\mu
$$
and hence the left-hand side of the inequality has to be finite. Conversely, assume that $f$ is as the statement says. By Lemma~\ref{equivalence1}, we know that there exists a sequence $\{s_n\}_{n \in \N}$ of simple measurable maps that converges almost everywhere pointwise to $f$ and $\|s_n (\om)\| \leq 2\|f(\om)\|$ almost everywhere. By the dominated convergence theorem and the assumption that $\|f\|$ is integrable, we obtain
$$
\int_\Om \|s_n - f\| \, d\mu \rightarrow 0.
$$

Observe that if $\{s_n\}_{n \in \N}$ is a sequence like in part (a), then the sequence of the $\int_\Om s_n \, d\mu$'s is a Cauchy sequence. Since $B$ is complete, it converges and part (b) is proved.

Finally, if $\{s_n\}_{n \in \N}$ and $\{t_n\}_{n \in \N}$ are two sequences like in part (a), then
$$
\left\| \int_\Om (s_n - t_n) \, d\mu\right\| \leq \int_\Om \| s_n - f \| \,d\mu + \int_\Om \| t_n - f \|\, d\mu.
$$
Since the right-hand side converges to 0, part (c) is proved.
\end{proof}

\begin{defn}
 Let $(\Om, \mathcal{F}, \mu)$ be a measure space and $B$ a Banach space. In view of Theorem~\ref{integrables}, a map $f : \Om \rightarrow B$ is \emph{Bochner integrable}\index{Bochner integrable map} if $f$ is measurable, separable and $\int_\Om \|f\| \, d\mu$ is finite. Its integral is defined by
$$
\int_\Om f \, d\mu := \lim_{n \to \infty} \int_\Om s_n \, d\mu,
$$
where $\{s_n\}_{n \in \N}$ is any sequence of simple integrable maps such that $\int_\Om \|s_n -f\| \, d\mu$ converges to 0. We also define $\mathcal{L}^1_B(\mu)$\index{L one@$\mathcal{L}^1_B(\mu)$} to be the set of all the Bochner integrable maps from $\Om$ to $B$ with respect to the measure $\mu$.
\end{defn}

It will be convenient for the subsequent discussion to make the following definition.

\begin{defn}
 Let $(\Om, \mathcal{F}, \mu)$ be a measure space and $B$ a Banach space. If $f : \Om \rightarrow B$ is a Bochner integrable map, then a pair $(\{s_n\}_{n \in \N}, N)$ is a \emph{defining sequence for (the integral of)}\index{defining sequence} $f$ if $N$ is a null subset of $\Om$ and $\{s_n\}_{n \in \N}$ is a sequence of simple maps fulfilling three conditions. Firstly, 
$$
\lim_{n \to \infty} \int_\Om s_n \, d\mu = \int_\Om f \, d\mu.
$$
Secondly, the sequence converges pointwise to $f$ on $N^c$ and thirdly $\|s_n(\om)\| \leq 2 \|f(\om)\|$ on $N^c$.
\end{defn}
The definition of a Bochner integrable map and Theorem~\ref{integrables} say that every integrable map admits a defining sequence. Before ending this section, we give an efficient way to work with the Bochner integral in the case where the measurable space is a locally compact group and the functions to integrate are continuous with compact support.

\begin{rmq}\label{integrale_C_c(G):rmk}
Let $G$ be a locally compact group endowed with its left Haar measure $\lambda$ and $B$ a Banach space. Assume that $f$ is a map belonging to $C_c(G, B)$ and let $K$ be its compact support. For $b$ in $B$, let $B(b, 1/n)$ be the open ball of radius $1/n$ around $b$. By compactness, for every integer $n$ we can find $g_1^n, \dots, g_{m_n}^n$  in $K$ such that
$$
K \subseteq B^n_1 \cup \cdots \cup B^n_m,
$$ 
where $B^n_i = \inv{f}(B(f(g^n_i), 1/n))$ for every $i$ in $\{1, \dots, m_n\}$. Define
$$
A^n_1 := B^n_1 \quad \text{ and } \quad A^n_i := B^n_i \backslash \bigcup_{1 \leq j < i} A^n_j, \quad (n \in \N, 2 \leq i \leq m_n).
$$
Notice that these sets are measurable because $f$ is continuous. If for every integer $n$ we define a simple map $s_n : G \rightarrow B$ by
$$
s_n := \sum_{i=1}^{m_n} \indicc{A^n_i} f(g^n_i),
$$
then the $s_n$'s converge pointwise to $f$. This shows that $f$ is both separable and measurable. Since the $s_n$'s are norm dominated by the $L^1$ function $g \mapsto \indicc{K}(g) \sup_{g \in K}\|f(g)\|$, the dominated convergence theorem as well as parts (b) and (c) of Theorem~\ref{integrables} can be used to obtain the integrability of $f$ and the equality
\begin{equation}\label{eq_rmq_int}
\int_G f \, d\lambda = \lim_{n \to \infty} \sum_{i = 1 }^{m_n} \lambda (A^n_i) f(g^n_i).
\end{equation}
\end{rmq}

\section{Properties of the Bochner integral}

In this section, it will be established in a row and without further comments four results about the Bochner integral. The first one is a collection of facts about Bochner integrable maps. The second one says that the integral commutes with bounded linear maps. The third and fourth ones are the dominated convergence theorem and the Fubini's theorem for the Bochner integral.

\begin{thm}\label{basic_properties_integral:thm}
Let $(\Om, \mathcal{F}, \mu)$ be a measure space and $B$ a Banach space.
 \begin{itemize}
  \item[(a)]The set $\mathcal{L}^1_B(\mu)$ of Bochner integrable maps from $\Om$ to $B$ is a vector space over $\R$ and $\|f \| := \int_\Om \|f\| \, d\mu$ is a semi-norm on $\mathcal{L}^1_B(\mu)$.
  \item[(b)]For every $f$ in $\mathcal{L}^1_B(\mu)$, the inequality $\left\|\int_\Om f \, d\mu \right\| \leq \int_\Om \|f\| \, d\mu$ holds.
  \item[(c)]The Bochner integral is a bounded linear map from $\mathcal{L}^1_B(\mu)$ to $B$ with respect to the semi-norm defined in (a).
 \end{itemize}
\end{thm}

\begin{proof}
 Let $f, g : \Om \rightarrow B$ be measurable and separable maps. We already know from Lemma~\ref{lemme1:30.11} that the linear combination $\alpha f + \beta g$ is measurable and separable for every real numbers $\alpha$ and $\beta$. If moreover $\|f\|$ and $\|g\|$ are integrable, then so is $\|\alpha f + \beta g\|$. Therefore, linear combinations of integrable maps are integrable. The statement about the semi-norm being clear, part (a) is now over.

In order to prove part (b), let $\{s_n\}_{n \in \N}$ be a sequence as in Theorem~\ref{integrables}. We have
$$
\left\| \int_\Om f \, d\mu \right\| = \lim_{n \to \infty} \left\| \int_\Om s_n \, d\mu \right\| \leq \lim_{n \to \infty} \int_\Om \|s_n\| \, d\mu  \leq \int_\Om \|f\|\, d\mu,
$$
where the last inequality comes from $\|s_n (\om)\| \leq \|s_n (\om)-f(\om)\|+\|f(\om)\|$ and from the choice of the $s_n$'s.

Finally, part (c) is almost obvious. First, we observe that the integral is linear on simple integrable maps. From this observation, we obtain the linearity for general integrable maps. The definition of the semi-norm on $\mathcal{L}^1_B(\mu)$ turns the integral into a norm-one linear map.
\end{proof}

\begin{thm}\label{commute_lineaires}
 Let $f$ be a map from a measure space $(\Om, \mathcal{F}, \mu)$ to a Banach space $B$. Let $C$ be another Banach space and $F : B \rightarrow C$ a continuous linear map. If $f$ is Bochner integrable, then so is $F \circ f$ and
\begin{equation}\label{int_commut}
F\left( \int_\Om f \, d\mu \right) = \int_\Om F \circ f \, d\mu.
\end{equation}
\end{thm}

\begin{proof}
Assume that $f$ is Bochner integrable. Let us first show that $F \circ f$ is measurable and separable. By Lemma~\ref{equivalence1}, it is enough to show that there exists a sequence of simple maps converging pointwise to $F \circ f$ almost everywhere. Since $f$ is measurable and separable, let $\{s_n\}_{n \in \N}$ be a sequence of simple maps given by Lemma~\ref{equivalence1}. Our assumptions imply that the $F \circ s_n$'s form a sequence of simple maps converging pointwise to $F \circ f$ almost everywhere and, by Lemma~\ref{equivalence1}, this in turn implies measurability and separability of $F \circ f$. Since $\| F \circ f (\om)\| \leq \|F\| \|f(\om)\|$ for every $\om$ in $\Om$,
$$
\int_\Om \|F \circ f\| \, d\mu \leq \|F\| \int_\Om \| f \| \, d\mu.
$$
This inequality, the integrability of $f$ and the fact that $F \circ f$ is measurable and separable show that $F \circ f$ is Bochner integrable.

The equation~\eqref{int_commut} holds whenever $f$ is a simple map. Using the definition of the integral for general integrable maps and the continuity of $F$, we see that it holds for every integrable map $f$.
\end{proof}

\begin{thm}[Dominated convergence\index{dominated convergence}]\label{convergence_dominee}
 Let $\{f_n\}_{n \in \N}$ be a sequence of measurable and separable maps from a measure space $(\Om, \mathcal{F}, \mu)$ to a Banach space $B$ converging pointwise to a map $f$ almost everywhere. If there exists a Lebesgue integrable function $g : \Om \rightarrow \R_+ \cup \{\infty\}$ dominating in norm the $f_n$'s almost everywhere, then all of the $f_n$'s and $f$ are Bochner integrable and
\begin{equation}\label{eq2:30.11}
 \lim_{n \to \infty} \int_\Om f_n \, d\mu = \int_\Om f \, d\mu.
\end{equation}
\end{thm}

\begin{proof}
 Assume that there exists a function $g$ like in the statement of the theorem and let $N$ be a null subset of $\Om$ such that
$$
\|f_n (\om) \| \leq g(\om) \quad \text{ and } \quad f_n(\om) \rightarrow f(\om),
$$
for every $\om$ in $N^c$. Since $g$ is integrable, the domination condition implies that all of the $f_n$'s are Bochner integrable. By Lemma~\ref{lemma1:29.11}, $f$ is measurable and separable. Since $\|f\| \leq g$ almost everywhere, $f$ is integrable. The equation~\eqref{eq2:30.11} follows from the standard dominated convergence theorem because
$$
\left\| \int_\Om f_n \, d\mu -  \int_\Om f \, d\mu \right\| \leq \int_\Om \|f_n - f\| \, d\mu
$$
and $\|f_n - f\| \leq 2g$ almost everywhere.
\end{proof}

Before stating Fubini's Theorem for the Bochner integral, let us remind some useful facts about classic measure and integration theory.

\begin{lem}\label{lem_Fub1}
 Let $(\Om, \mathcal{F}, \mu)$ and $(\Phi, \mathcal{F}', \nu)$ be two $\sigma$-finite measure spaces. For any subset $A$ of $\Om \times \Phi$ and any pair $(\om, \phi)$ in $\Om \times \Phi$ define
$$
S_\Om(\phi, A):=\{\om \in \Om : (\om, \phi) \in A\} \quad \text{ and } \quad S_\Phi(\om, A):=\{\phi \in \Phi : (\om, \phi) \in A\}.
$$
If $A$ belongs to the product $\sigma$-field $\mathcal{F}\otimes\mathcal{F}'$, then for every $(\om, \phi)$ in $\Om \times \Phi$ the sets $S_\Om(\phi, A)$ and $S_\Phi(\om, A)$ are measurable. Moreover, the maps
$$
\begin{array}{lcl}
\Phi \longrightarrow \R_+ & \text{ and } & \Om \longrightarrow \R_+ \\
\phi \longmapsto \mu(S_\Om(\phi, A)) &  \quad &\om \longmapsto \nu(S_\Phi(\om, A))
\end{array}
$$
are measurable too.
\end{lem}

\begin{thm}[Classic Fubini's Theorem]\label{classic_Fubini}
 Let $(\Om, \mathcal{F}, \mu)$ and $(\Phi, \mathcal{F}', \nu)$ be two $\sigma$-finite measure spaces. If $f : \Om \times \Phi \rightarrow \R$ belongs to $\mathcal{L}^1_\R(\mu \otimes \nu)$, then the following holds.
\begin{itemize}
  \item[(a)]The map $f(\om, \cdot)$ belongs to $\mathcal{L}^1_\R(\nu)$ for almost every $\om$ in $\Om$ and therefore we can almost everywhere define a map $If : \Om \rightarrow \R$ by declaring that $If(\om) = \int_\Phi f(\om, \cdot) \, d\nu$.
  \item[(b)]The map $If$ almost everywhere defined in part (a) belongs to $L^1_\R(\mu)$.
  \item[(c)]We have
$$
\int_{\Om \times \Phi} f \, d(\mu\otimes\nu) = \int_\Om If \, d\mu = \int_\Om \left( \int_\Phi f(\om, \phi) \, d\nu(\phi)\right) \, d\mu(\om).
$$
  \item[(d)]The three preceding parts remain true if we switch the roles of $\Om$ and $\Phi$.
\end{itemize}
\end{thm}

Here is Fubini's Theorem for the Bochner integral.

\begin{thm}[Fubini's Theorem\index{Fubini's Theorem}]\label{Fubini}
 Let $(\Om, \mathcal{F}, \mu)$ and $(\Phi, \mathcal{F}', \nu)$ be two $\sigma$-finite measure spaces. If $f$ belongs to $\mathcal{L}^1_B(\mu \otimes \nu)$, where $B$ is a Banach space and $\mu \otimes \nu$ the product measure on $\Om \times \Phi$, then the following statements hold.
\begin{itemize}
  \item[(a)]For almost every $\om $ in $ \Om$, the map $f(\om, \cdot)$ is $\nu$-Bochner integrable.
  \item[(b)]For almost every $\phi $ in $ \Phi$, the map $f( \cdot , \phi)$ is $\mu$-Bochner integrable.
  \item[(c)]We have the formula
\begin{eqnarray*}
 \int_{\Om \times \Phi} f \, d(\mu\otimes\nu) &=& \int_\Om \left(\int_\Phi f(\om, \phi) \, d\mu(\om) \right)\, d\nu(\phi) \\
					  &=&	\int_\Phi \left(\int_\Om f(\om, \phi) \, d\nu(\phi) \right)\, d\mu(\om).
\end{eqnarray*}
\end{itemize}
\end{thm}

The proof of Theorem~\ref{Fubini} can be split in two lemmas.

\begin{lem}\label{point(a)}
 Part (a) of Theorem~\ref{Fubini} holds.
\end{lem}
\begin{proof}
 Let $f$ belong to $\mathcal{L}^1_B(\mu \otimes \nu)$ and $N$ be a $(\mu \otimes \nu)$-null set such that $f(N^c)$ is contained in the closure of a countable subset $Y$ of $B$.

Using the notation of Lemma~\ref{lem_Fub1}, the slices $S_\Phi(\om, N)$ are $\nu$-null sets for almost every $\om$ in $\Om$. Hence, for almost every $\om$, the image $f(\om, \Phi \backslash S_\Phi(\om, N))$ is contained in the closure of $Y$. This shows that $f(\om, \cdot)$ is separable for almost every $\om$.

Using Lemma~\ref{lem_Fub1}, we see that for every measurable function $h$ from $\Om \times \Phi$ into any measure space the function $h(\om, \cdot)$ is measurable for every $\om$ in $\Om$. In particular, $f(\om, \cdot)$ is measurable.

It must still be verified that $\|f(\om, \cdot)\|$ is $\nu$-integrable for almost every $\om$ to obtain part (a) of Theorem~\ref{Fubini}. This can be achieved by means of the classic Fubini's Theorem~\ref{classic_Fubini}, because $\|f(\om, \cdot)\|$ belongs to $L^1_\R(\nu)$.
\end{proof}

\begin{lem}\label{lem_Fub3}
 Let $\Om, \Phi$ and $f$ be as in the statement of Fubini's Theorem~\ref{Fubini}. If $(\{s_n\}_{n \in \N}, N)$ is a defining sequence for $f$, then the following properties hold.
\begin{itemize}
  \item[(a)]For almost every $\om$ in $\Om$, the sequence $\{s_n(\om, \cdot)\}_{n \in \N}$ of maps from $\Phi$ to $B$ converges pointwise to $f(\om, \cdot)$ almost everywhere on $\Phi$.
  \item[(b)]The maps $Is_n : \Om \rightarrow B$ defined by $Is_n(\om) := \int_\Phi s_n(\om, \cdot) \, d\nu$ converge  pointwise to $If(\om):=\int_\Phi f(\om, \cdot) \, d\nu$ almost everywhere.
  \item[(c)]The maps $Is_n$ defined above are $\mu$-Bochner integrable.
  \item[(d)]The three preceding parts remain true if we switch the roles of $\Om$ and $\Phi$.
\end{itemize}
\end{lem}
\begin{proof}
Let $(\{s_n\}_{n \in \N}, N)$ be a defining sequence for $f$. If we use the notation of Lemma~\ref{lem_Fub1}, the sets $S_\Phi(\om, N)$ are of $\nu$-null sets for almost every $\om$ in $\Om$. For such a $\om$, the functions $s_n(\om, \cdot)$ converge pointwise to $f(\om, \cdot)$ on $S_\Phi(\om, N)^c$. In other words, for almost every $\om$ in $\Om$, the $s_n(\om, \cdot)$'s converge to $f(\om, \cdot)$ almost everywhere on $\Phi$. This proves part (a).

Once we realize that the maps $Is_n$ and $If$ are well defined (by using Lemma~\ref{point(a)}), part (b) is proved by an application of dominated convergence Theorem~\ref{convergence_dominee}.

Let us prove part (c). First, we show that the $Is_n$'s are measurable and separable. If $s_n$ is a finite sum of the form
$$
s_n = \sum_k \indicc{A_k} \, y_k,
$$
then $Is_n(\om) = \sum_k \mu(S_\Phi(\om, A_k)) y_k$. By Lemma~\ref{lem_Fub1}, the previous equality implies that $Is_n$ is a finite sum of measurable and separable maps.  It now follows from part (b) of Lemma~\ref{lemme1:30.11} that the $Is_n$'s are measurable and separable. Because we have the estimate $\|s_n (\om, \phi)\| \leq 2 \|f(\om, \phi)\|$, one can see that
$
\int_\Om \left\| \int_\Phi s_n(\om, \cdot) \, d\nu \right\| \, d\mu
$
is finite.
\end{proof}

With these lemmas, the proof of Fubini's Theorem becomes straightforward.

\begin{proof}[Proof of Thm~\ref{Fubini}]
Point (a) is Lemma~\ref{point(a)} and part (b) is symmetric to part (a). To prove part (c), it is enough to show that the equality
$$
\int_{\Om \times \Phi} f \, d(\mu\otimes\nu) = \int_\Om \left(\int_\Phi f(\om, \phi) \, d\nu(\phi) \right) \, d\mu(\om)
$$
holds. If $f = \indicc{A} \, y$, where $A$ belongs to $\mathcal{F} \otimes \mathcal{F}'$ and $y$ belongs to $Y$, then Theorem~\ref{commute_lineaires} gives
$$
\int_{\Om \times \Phi} f \, d(\mu \otimes \nu) = \left(\int_{\Om \times \Phi} \indicc{A} \, d(\mu \otimes \nu) \right)y.
$$
By classic Fubini's Theorem, the right-hand side is equal to
$$
\left(\int_\Om \int_\Phi \indicc{A}(\om, \phi) \, d\nu(\phi) \, d\mu(\om)\right) y = \int_\Om \int_\Phi f(\om, \nu) \, d\nu(\phi) \, d\mu(\om)
$$
and so the result holds whenever $f$ is the product of the characteristic function of a measurable set and a constant vector in $B$. Using the linearity of the Bochner integral, we see that the result also holds for simple maps.

If $f$ belongs to $\mathcal{L}^1_B(\mu \otimes \nu)$, let $(\{s_n\}_{n \in \N}, N)$ be a defining sequence for $f$. We have
$$
\int_{\Om \times \Phi} f \, d(\mu \otimes \nu) = \lim_{n \to \infty} \int_{\Om \times \Phi} s_n \, d(\mu \otimes \nu) = \lim_{n \to \infty} \int_\Om \left( \int_\Phi s_n(\om, \phi) \, d\nu(\phi) \right) \,d\mu(\om).
$$
Using parts (b) and (c) of Lemma~\ref{lem_Fub3} and the dominated convergence theorem in order to switch the integral over $\Om$ and the limit, we have now reached the end.

\end{proof}

\chapter{$G$-spaces and Banach modules}\label{modules:Ch}
  This chapter is essentially a collection of definitions and facts about $G$-spaces and Banach $G$-modules. About the presented matters, sources are~\cite{These_Monod}, \cite{AVENIR} and~\cite{Folland}. At the end of the chapter, a theorem due to Anker~\cite{Anker} will be restated in a more general setting than its original one.

For the rest of the chapter, $G$ will denote a locally compact topological group. Whenever notions involving a measure on $G$ will be brought up, for example the space $L^1G$, they are to be understood relatively to the left Haar measure on $G$, which will be denoted by $\lambda$. Moreover, when the word \emph{action} will be used without further details, let us agree that it consists in a left action.

Let us mention here that $C_c(G)$\index{Continuous c@$C_c(G)$} denotes the set of all compactly supported continuous functions on $G$.

\section{$G$-spaces}

By a $G$-\emph{space}\index{G space@$G$-space}, we mean a topological space on which the group $G$ acts. The epithet \emph{left} is added when the action is from the left. The actions will be denoted by $\cdot$ or by simple juxtaposition. In other words, the result of the action of an element $g$ in $G$ on an element $x$ belonging to some left $G$-space is $g \cdot x$ or simply $gx$. If $X$ and $B$ are left $G$-spaces, a left action on $B^X$ can be defined by declaring that
$$
(g \cdot f)(x)= g \cdot f(\inv{g} \cdot x) \quad (x \in X, g \in G),
$$
for every map $f : X \rightarrow B$. We will call this action the \emph{diagonal action}\index{diagonal action} of $G$ on $B^X$. An action on a $G$-space $X$ is \emph{continuous}\index{continuous action} if the map
$$
\begin{array}{ccl}
 G \times X & \longrightarrow & X \\
(g, x)& \longmapsto & g \cdot x
\end{array}
$$
is continuous. If $X$ is a $G$-space with continuous action, we say that $X$ is a \emph{continuous} $G$-\emph{space}\index{continuous G space@continuous $G$-space}. Our first interesting result concerns the regularity of the diagonal action. It requires the statement of two lemmas.

\begin{lem}\label{lem1_diag_action}
 Let $(B, d)$ be a continuous $G$-metric space. For every $g$ in $G$, every compact subset $A$ of $B$ and every $\epsilon >0$, there exists a neighborhood $U$ of $g$ such that
$$
d(h  a, g a) < \epsilon \quad ( h \in U, a \in A).
$$
\end{lem}
\begin{proof}
 Fix $g$ in $G$, $\epsilon > 0 $ and a compact subset $A$ of $B$. For every $b$ in $A$, by the continuity of the action at $(g, b)$, there exists open neighborhoods $U_b$ and $V_b$  of $g$, respectively $b$, such that
$$
d(ha, ga) \leq d(ha, gb) + d(gb, ga) < \epsilon \quad (h \in U_b, a \in V_b).
$$
Let $b_1, \dots, b_n$ be elements of $A$ such that
$
A \subseteq V_{b_1} \cup \cdots \cup V_{b_n}.
$
The neighborhood of $g$ defined by $U := U_{b_1} \cap \cdots \cap U_{b_n}$ is such that
\begin{equation*}
d(ha, ga) < \epsilon \quad (h \in U, a \in A). \qedhere
\end{equation*}
\end{proof}

\begin{lem}\label{lem2_diag_action}
 Let $X$ be a compact $G$-space and endow $X$ with the unique uniform structure compatible with its topology. If the action of $G$ on $X$ is continuous, then the family of maps $\{f_x\}_{x \in X}$ from $G$ to $X$ defined by $f_x(g):=g \cdot x$ is equicontinuous.
\end{lem}
\begin{proof}
First, some notations must be set. For a subset $V$ of $X \times X$, let $V[x]$ be the set of all the elements $y$ of $X$ such that $(y, x)$ belongs to $V$.

 Assume that $X$ is a continuous $G$-space and fix $g$ in $G$. Let us show that the family of maps given above is equicontinuous at $g$.  Choose a neighborhood of the diagonal $V$ in $X \times X$ and recall that for every $x$ in $X$, the set $V[x]$ is a neighborhood of $x$. We will show that $\bigcap \inv{f_x}(V[f_x(g)])$, where the intersection ranges over all the $x$'s in $X$, is a neighborhood of $g$. For every $x$ in $X$, let $U_x$ and $V_x$ be open neighborhood of $g$ and $x$, respectively, such that
$$
f_y(h) = h \cdot y \in V[f_x(g)] \quad (h \in U_x, y \in V_x).
$$
The compactness of $X$ implies that there exists $x_1, \dots, x_n$ such that
$$
X \subseteq V_{x_1} \cup \cdots \cup V_{x_n}.
$$
Setting $U := U_{x_1} \cap \cdots \cap U_{x_n}$, we obtain that $f_x(h)$ belongs to $V[f_x(g)]$ for every $h$ in $U$ and every $x$ in $X$. This means that $\bigcap \inv{f_x}(V[f_x(g)])$ contains the neighborhood $U$.
\end{proof}

\begin{prop}\label{action_diag}
 Let $X$ be a compact $G$-space, $B$ a normed $G$-vector space and endow $C(X, B)$ with the uniform norm. If the actions of $G$ on $X$ and $B$ are continuous and if the latter is by isometries, then the diagonal action of $G$ on $C(X, B)$ is norm continuous.
\end{prop}
\begin{proof}
 Let $\{g_n\}_{n \in N}$ and $\{f_m\}_{m \in M}$ be nets in $G$ and $C(X, B)$ converging respectively to $g$ and $f$. We estimate
$$
\|g_n \cdot f_m - g \cdot f\| \leq \|g_n \cdot f_m - g_n \cdot f \|+\| g_n \cdot f - g \cdot f\|.
$$
Since $G$ acts isometrically on $B$, it also acts isometrically on $C(X, B)$ and so the first term in the right-hand side tends to $0$. The second one is bounded from above by
$$
\sup_{x \in X} \|g_n \cdot f(\inv{g_n} \cdot x) - g_n \cdot f(\inv{g} \cdot x)\| + \sup_{x \in X} \|g_n \cdot f(\inv{g} \cdot x) - g \cdot f(\inv{g} \cdot x)\|
$$
which is equal to
$$
\sup_{x \in X} \| f(\inv{g_n} \cdot x) -  f(\inv{g} \cdot x)\| + \sup_{x \in X} \|g_n \cdot f(\inv{g} \cdot x) - g \cdot f(\inv{g} \cdot x)\|.
$$
The map $f$ being continuous on a compact space, it is uniformly continuous and so, by Lemma~\ref{lem2_diag_action}, the first term in the above expression tends to $0$. The second one also tends to $0$ because of Lemma~\ref{lem1_diag_action} and the fact that $f(X)$ is a compact subset of $B$.
\end{proof}

\subsection*{The space $L^1G$}

We want to apply Proposition~\ref{action_diag} to $B = L^1G$. That is to say, we want to show that for every compact continuous $G$-space $X$ the diagonal action of $G$ on $C(X, L^1G)$ is norm continuous.

\begin{prop}\label{uniformement_continue:prop}
 If $f$ belongs to $C_c(G)$, then $\|y \cdot f - f\|_\infty$ tends to zero as $y$ tends to $1_G$.
\end{prop}
\begin{proof}
 Let $f$ belong to $C_c(G)$ and fix $\epsilon > 0$. If $K$ denotes the compact support of $f$, then for every $x$ in $K$ there is a neighborhood $U_x$ of $1_G$ such that $|f(\inv{y} x) - f(x)| < \epsilon/2$ for every $y$ in $U_x$ and there is a symmetric neighborhood $V_x$ of $1_G$ such that $V_x V_x$ is contained in $U_x$. By compacity of $K$, there exists $x_1, \dots, x_n$ in $K$ such that
$$
K \subseteq  V_{x_1} x_1 \cup \cdots \cup  V_{x_n}x_n.
$$
Define a neighborhood of $1_G$ by $V := V_{x_1}  \cap \cdots \cap  V_{x_n}$. We claim that $\|y \cdot f - f\|_\infty < \epsilon$ for every $y$ in $V$.

To show this, let $y$ belong to $V$. If $x$ belongs to $K$, then there is some $j$ for wich $x$ belongs to $V_{x_j} x_j$ and so $x = \inv{y_j} x_j$ for some $y_j$ in $V_{x_j}$. We have $\inv{y} x =\inv{y} \inv{y_j} x_j = y' x_j$ with $y'$ in $U_{x_j}$. This implies that
$$
|f(\inv{y} x) - f(x)| \leq |f(y' x_j) - f(x_j)| + |f(x_j) - f(\inv{y_j} x_j)| \leq \epsilon.
$$
Similarly, if $\inv{y} x$ belongs to $K$, then $|f(\inv{y} x) - f(x)| < \epsilon$. If neither $x$ nor $\inv{y} x$ belong to $K$, then $f(x) = f(\inv{y}x) = 0$.
\end{proof}

\begin{prop}\label{continue_sur_L1:prop}
If $f$ belongs to $L^1G$, then $\|y \cdot f - f \|_1$ tends to $0$ as $y$ tends to $1_G$.
\end{prop}
\begin{proof}
 By density of $C_c(G)$ in $L^1G$, it is enough to prove the result for $f$ in $C_c(G)$. Assume that $f$ belongs to $C_c(G)$ and fix a compact neighborhood $V$ of $1_G$. Recall that since the multiplication from $G \times G$ to $G$ is continuous, the product $K_1 K_2$ of two compact subsets of $G$ is compact. In particular, the subset $K :=  V (\mathrm{supp} f)$ is compact and therefore of finite measure. Notice that if $y$ belongs to $V$, then $\mathrm{supp}(y \cdot f)$ is contained in $K$ and so $\|y \cdot f - f\|_1 \leq \lambda(K) \|y \cdot f - f\|_\infty$. Using Proposition~\ref{uniformement_continue:prop} and this estimate, we are done.
\end{proof}

\begin{cor}\label{action_G_L1_cont:cor}
 The action of $G$ on $L^1G$ is continuous.
\end{cor}
\begin{proof}
 Let $\{g_n\}_{n \in N}$ and $\{f_m\}_{m \in M}$ be two nets converging respectively to $g$ in $G$ and $f$ in $L^1G$, respectively. We have
$$
\|g_n \cdot f_m - g \cdot f\|_1 \leq \|g_n \cdot f_m - g_n \cdot f\|_1 + \|g_n \cdot f - g \cdot f\|_1.
$$
Since $G$ acts isometrically on $L^1G$, the first term in the right-hand side is equal to $\| f_m -  f\|_1$, which tends to zero. The second term being equal to $\|\inv{g}g_n \cdot f -   f\|_1$, Proposition~\ref{continue_sur_L1:prop} shows that it also tends to zero.
\end{proof}

\begin{cor}\label{diag_action_L^1_continuous:cor}
 If $X$ is a compact continuous $G$-space, then the diagonal action on $C(X, L^1G)$ is continuous.
\end{cor}
\begin{proof}
 The result is a consequence of Corollary~\ref{action_G_L1_cont:cor} and Proposition~\ref{action_diag}.
\end{proof}

%
%
\section{Banach $G$-modules}
We introduce now the Banach $G$-modules and some associated constructions. In particular, we introduce $(G, X)$-modules and the $L^1G$-module structure of continuous Banach $G$-modules.
\subsection*{Terminology}
A \emph{Banach} $G$-\emph{module}\index{Banach G module@Banach $G$-module} is a pair $(\pi, E)$, where $E$ is a Banach space and $\pi$ a group homomorphism from $G$ to the group of isometric linear automorphisms of $E$. A Banach $G$-module $(\pi, E)$ will be merely referred as the Banach space $E$. For $g$ in $G$ and $v$ in $E$ we will write $g \cdot v$ or even $g v$ for $\pi(g)v$. A Banach $G$-module $(\pi, E)$ is \emph{continuous}\index{continuous Banach G module@continuous Banach $G$-module} if the map
$$
\begin{array}{ccl}
 G \times E & \longrightarrow & E\\
(g, v) & \longmapsto & \pi(g)v
\end{array}
$$
is continuous.

\begin{exple}\label{C(X, L1)_banach_mod:exple}
If $X$ is a compact space, then $C(X, L^1G)$ endowed with the uniform norm is a continuous Banach $G$-module for the diagonal action. This is mainly a consequence of Corollary~\ref{diag_action_L^1_continuous:cor}.
\end{exple}

Let $X$ be a compact space and endow $C(X)$ with the uniform norm. A Banach space $E$ is a $C(X)$-\emph{module}\index{C(X) module@$C(X)$-module} if $E$ is an algebraic $C(X)$-module and if the inequality $\|\phi v\| \leq \|\phi\| \|v\|$ holds for every $\phi$ in $C(X)$ and every $v$ in $E$. A $C(X)$-module $E$ is of \emph{type C}\index{type C} if the inequality
$$
\|\phi_1 v_1 + \cdots + \phi_n v_n\| \leq \|\phi_1 + \cdots + \phi_n\| \cdot \max_{1 \leq i \leq n} \|v_i\|
$$
holds for every $v_1, \dots, v_n$ in $E$ and every $\phi_1, \dots, \phi_n \geq 0$ in $C(X)$. For the same $\phi_i$'s, the module $E$ is of \emph{type M}\index{type M} if the inequality
$$
\|\phi_1 v \| + \cdots \|\phi_n v\| \leq \|\phi_1 + \cdots + \phi_n\| \cdot \|v\|
$$
holds for every $v$ in $E$. Notice that if $E$ is a $C(X)$-module of type C, then the inequality
\begin{equation}\label{typeC:eq}
\|\phi_1 v_1 + \cdots + \phi_n v_n\| \leq  \||\phi_1| + \cdots + |\phi_n|\| \cdot \max_{1 \leq i \leq n} \|v_i\|
\end{equation}
is always verified. This is seen by writing $\phi \cdot v = \phi^+ \cdot v + \phi^- \cdot (-v)$, for $\phi$ in $C(X)$ and $v$ in $E$.

The concept of type M module enables us to prove the following lemma.

\begin{lem}\label{dualTypeC}
Let $E$ be a $C(X)$-module.
\begin{itemize}
  \item[(a)]$E^*$ is of type C if and only if $E$ is of type M.
  \item[(b)]$E^*$ is of type M if and only if $E$ is of type C.
  \item[(c)]If $E$ is of type C (respectively M), then so is its double dual.
\end{itemize}
\end{lem}
\begin{proof}
 Part (c) is obvious using parts (a) and (b). To prove (a) and (b), first observe that $E^*$ is of type C (respectively M) if $E$ is of type M (respectively C). To show that the converse also holds, notice that conditions C and M descend to submodules, embed $E$ into $E^{**}$ and repeat the argument above.
\end{proof}

Let $X$ be a compact $G$-space. The notions of Banach $G$-modules and $C(X)$-modules are now combined. A Banach space $E$ is a $(G, X)$-\emph{module}\index{G X module@$(G, X)$-module} if it is a Banach $G$-module as well as a $C(X)$-module and if the consecutive actions of $\inv{g}, \phi$ and $g$ correspond to the action of $g \cdot \phi$. In other words, the formula
\begin{equation}\label{C(X)compatibility}
g \cdot( \phi \cdot(\inv{g} \cdot v)) = (g \cdot \phi) \cdot v
\end{equation}
is asked to hold for every $g$ in $G$, every $\phi$ in $C(X)$ and every $v$ in $E$.

\begin{exple}\label{ExC(X)ModuldeTypeC}
 Let $X$ be a compact space and $B$ a Banach space. If we endow $C(X)$ and $C(X, B)$ with their uniform norms, then $C(X, B)$ is a $C(X)$-module of type C. In fact, if $\phi_1, \dots, \phi_n$ are nonnegative functions in $C(X)$ and $v_1, \dots, v_n$ belong to $C(X, B)$, then the inequality
$$
\sup_{x \in X} \|\phi_1(x) v_1(x) + \cdots +\phi_n(x) v_1(x)\|_B \leq \sup_{x \in X} \{\phi_1(x) \|v_1\| + \cdots + \phi_n(x) \|v_n\|\} 
$$
can easily be used to obtain
$$
\|\phi_1 v_1 + \cdots + \phi_n v_n\| \leq \|\phi_1 + \cdots + \phi_n\|  \cdot \max_{1 \leq i \leq n} \| v_i \|.
$$
Moreover, if $X$ is a continuous $G$-space and $B$ a continuous Banach $G$-module, then a careful and generous use of parenthesis together with Proposition~\ref{action_diag} shows that $C(X, B)$ is even a $(G, X)$-module for the diagonal action. In particular, by Corollary~\ref{diag_action_L^1_continuous:cor}, the space $C(X, L^1G)$ is a $(G, X)$-module of type C for the diagonal action.
\end{exple}
%
%
%
%
\subsection*{The $L^1G$-module structure of continuous Banach $G$-modules}

Consider the ring $(L^1G, +, *)$, where $*$ denotes the convolution operator and let $E$ be a continuous Banach $G$-module. Recall that for every $\psi$ in $C_c(G)$ and every $v$ in $E$, the Bochner integral
\begin{equation}\label{L1ModuleStruct}
\int_G \psi(t)  (t \cdot v) \,d\lambda(t)
\end{equation}
is well defined (see Remark~\ref{integrale_C_c(G):rmk}). From now on, the integral~\eqref{L1ModuleStruct} will denoted by $\psi * v$. Since the linear map $\psi \mapsto \psi * v$ is bounded with respect to the $L^1$ norm on $C_c(G)$, the density of $C_c(G)$ in $L^1G$ implies that we can still define $\psi * v$ for $\psi$ in $L^1G$. In fact, for $\psi$ in $L^1G$, let $\{\varphi_n\}_{n \in \N}$ be a sequence of compactly supported continuous functions converging in $L^1G$ to $\psi$. Since the estimate $\|\varphi_m * v - \varphi_n *v\| \leq \|\varphi_m - \varphi_n\|_1 \|v\|$ holds for every $v$ in $E$, the sequence of the $(\varphi_n * v)$'s is Cauchy. We define $\psi * v := \lim_{n \to \infty} \varphi_n * v$ for every $v$ in $E$. The above estimate implies that $\psi * v$ does not depend on the choice of the sequence $\{\varphi_n\}_{n \in \N}$.

The only tricky part in checking that this indeed defines a $L^1G$-module structure on $E$ is to verify that the equality
\begin{equation}\label{eq1:7.12}
(\psi * \varphi) * v = \psi * (\varphi * v)
\end{equation}
holds for every $\psi$ and $\varphi$ in $L^1G$ and every $v$ in $E$. By means of Theorem~\ref{commute_lineaires} and Fubini's Theorem~\ref{Fubini}, we first see that equation~\eqref{eq1:7.12} holds for every $v$ in $E$, whenever $\psi$ and $\varphi$ belong to $C_c(G)$. Using the continuity of the convolution, we then see that it still holds for general $\psi$ and $\varphi$.

\begin{rmq}\label{rmq1}
 Above, the expression $\psi * v$ has been defined by a density argument. It must be stressed that the almost everywhere defined map $t \mapsto \psi(t)(t \cdot v)$ from $G$ to a continuous Banach $G$-module $E$ is also Bochner integrable for $\psi$ in $L^1G$ and so the integral~\eqref{L1ModuleStruct} makes sense with $\psi$ in $L^1G$. The integrability can be proved by showing that the map in question is almost everywhere the pointwise limit of a sequence of simple maps.
\end{rmq}

The next result states some properties of the $L^1G$-module structure defined above. In particular, it shows how the $L^1G$-module and $G$-space structures of a continuous Banach $G$-module are compatible.

\begin{prop}\label{properties_L1_module_struc:prop}
Let $(\pi, E)$ be a continuous Banach $G$-module endowed with the $L^1G$-module structure defined above. The $G$-space structures of $E$ and $L^1G$ are compatible with the $L^1G$-module structure of $E$ in the sense that if $\delta_g$ denotes the Dirac measure at $g$ in $G$, then the equalities
\begin{equation}\label{eq1:8.12}
 \psi * (g \cdot v) = (\psi * \delta_g) * v
\end{equation}
 and
\begin{equation}\label{eq2:8.12}
 g \cdot (\psi * v) = (\delta_g * \psi) * v,
\end{equation}
hold for every $\psi$ in $L^1G$ and every $v$ in $E$. Moreover, the estimate $\|\psi * v\| \leq \|\psi\| \|v\|$ holds for every $\psi$ in $L^1G$ and every $v$ in $E$ and therefore the $L^1G$-module operations are continuous.
\end{prop}
\begin{proof}
 Let $\Delta$ denote the modular function of $G$. For every $\psi$ in $L^1G$ and every $g$ in $G$, we have
$$
(\psi  * \delta_g)(t) = \psi(t \inv{g})\Delta(\inv{g})
$$
almost everywhere on $G$ and so, for $v$ in $E$,
\begin{eqnarray*}
(\psi * \delta_g) * v &=& \int_G \psi(t \inv{g}) \Delta(\inv{g}) (t \cdot v) \, d\lambda(t) \\
		      &=& \int_G \psi(t g \inv{g}) \Delta(\inv{g}) \Delta(g) (t g \cdot v) \, d\lambda(t)\\
		      &=& \int_G \psi(t) [t \cdot (g v)] \, d\lambda(t) = \psi * (g\cdot v),
\end{eqnarray*}
which is exactly equation~\eqref{eq1:8.12}. As for equation~\eqref{eq2:8.12}, it is a consequence of the equality $(\delta_g * \psi) = g \cdot \psi$ and of Theorem~\ref{commute_lineaires}. In fact, by permuting the following integral with the continuous linear map $\pi(g)$, we obtain
$$
g \cdot(\psi * v) =  \int_G \psi(t) \pi(g) (t \cdot v) \, d\lambda(t) = \int_G \psi(\inv{g}t) (t \cdot v) \, d\lambda(t) = (g \cdot \psi) * v.
$$

The estimate at the end of the statement follows from the fact that $G$ acts isometrically on $E$ and from part (b) of Theorem~\ref{basic_properties_integral:thm}.
\end{proof}

\subsection*{Relative injectivity}

Let $\eta : A \rightarrow B$ be a map between two Banach spaces. For simplicity, the terminology will be as follows : the map $\eta$ is a \emph{morphism}\index{morphism} if it is linear and continuous. If $A$ and $B$ are both Banach $G$-modules, $\eta$ is said to be a $G$-\emph{morphism}\index{G-morphism@$G$-morphism} if $\eta$ is a $G$-equivariant morphism.

\begin{defn}
A Banach $G$-module $E$ is \emph{relatively injective}\index{relatively injective} if for every injective $G$-morphism $\iota : A \rightarrow B$ of continuous Banach $G$-modules and every $G$-morphism $\alpha : A \rightarrow E$ there exists a $G$-morphism $\beta : B \rightarrow E$ with $\|\beta\| \leq \|\alpha\|$ such that $\beta \circ \iota = \alpha$, whenever there exists a morphism $\sigma : B \rightarrow A$ with $ \sigma \circ \iota =  Id_A$ and $\|\sigma\| \leq 1$. In other words, the following diagram is asked to commute.
$$
\xymatrix{ A \ar@{^{(}->}[rr]_\iota \ar[rd]_\alpha && B \ar@{.>}[ld]^\beta \ar@/_1pc/[ll]_\sigma \\ & E}
$$
\end{defn}

It is worth noting that relative injectivity is a normalized condition on Banach $G$-modules. In fact, an equivalent property could have been defined by omitting the norm condition on $\sigma$ and by asking the extension $\beta$ to verify $\|\beta\| \leq \|\sigma\| \|\alpha\|$. For later use, a lemma is made out of this observation.

\begin{lem}\label{equivDefRelInj}
 A Banach $G$-module $E$ is relatively injective if and only if for every injective $G$-morphism $\iota : A \rightarrow B$ of continuous Banach $G$-modules admitting a left inverse morphism $\sigma : B \rightarrow A$ and every $G$-morphism $\alpha : A \rightarrow E$ there exists a $G$-morphism $\beta : B \rightarrow E$ with $\|\beta\| \leq \|\sigma\| \|\alpha\|$ such that $\beta \circ \iota = \alpha$.
\end{lem}
\begin{proof}
 Let $E$ be a Banach $G$-module. If $E$ satisfies the condition in the statement, then it is relatively injective because relative injectivity is a particular case of this condition. Conversely, assume that $E$ is relatively injective and let an extension problem of the form
$$
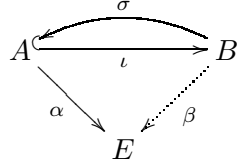
\xymatrix{ A \ar@{^{(}->}[rr]_\iota \ar[rd]_\alpha && B \ar@{.>}[ld]^\beta \ar@/_1pc/[ll]_\sigma \\ & E}
$$
be given, where $\iota$ and $\alpha$ are $G$-morphisms and $\sigma$ a left inverse morphism for $\iota$. There are two simple ways to obtain the desired $\beta$. The first one is to normalize $\sigma$ and $\iota$ in order to be in the setting of relative injectivity. The second one is to change the norm on $B$ to be in the setting of relative injectivity. Both ways lead to a $G$-morphism $\beta$ such that $\beta \circ \iota = \alpha$ and $\|\beta\| \leq \|\sigma\| \|\alpha\|$.
\end{proof}

%
%
\section{Anker's argument}\label{Anker:Sect}

First of all, the terminology will be as follows. If $E$ is a Banach $G$-module and $K$ is a subset of $G$, a vector $v$ in $E$ is said to be $(K, \epsilon)$-\emph{invariant}\index{invariant vector@invariant $(K, \epsilon)$-invariant vector} if $\|g \cdot v - v\| < \epsilon$ for every $g$ in $K$. The goal of this section is to restate in a more general setting than its original one the main result of~\cite{Anker} about $\epsilon$-invariant vectors. It must be stressed that even if our result is stated in a more general setting than the one in~\cite{Anker}, our proof is almost an exact copy of the original one.

We define $S(G)$\index{S(G)@$S(G)$} to be the subset of $L^1G$ containing all the nonnegative functions whose integral is equal to $1$. That is to say
\begin{equation}\label{SG}
S(G):=\{ \psi \in L^1G : \psi \geq 0 \text{ almost everywhere and } \int_G \psi \, d\lambda = 1\}.
\end{equation}
Using this notation, Property $(P_1)$\index{P1@$P_1$} is defined as follows : for every compact subset $K$ of $G$ and for every $\epsilon > 0$, there exists a $(K, \epsilon)$-invariant function $f$ in $S(G)$. Property $(P_1^*)$\index{P1*@$P_1^*$} is defined in the same way with $K$ finite instead of compact.

The main result of~\cite{Anker} is that Properties $(P_1)$ and $(P_1^*)$ are equivalent, whenever $G$ is a locally compact topological group. For the reader familiar to the subject : the originality of Anker's paper is not the result itself, which was already known for some years, but the way it is proved. This section is named after Anker because our proof is largely the same as his.

Let us put the notations above into a more general setting. Let $X$ be a compact space and define the property $(P_1(G, X))$\index{P1 of G, X@$P_1(G, X)$} as follows : for every compact subset $K$ of $G$ and for every $\epsilon > 0$, there exists a $(K, \epsilon)$-invariant map $f$ in $C(X, L^1G)$ such that $f(X)$ is contained in $S(G)$. Property $(P_1^*(G, X))$\index{P1* of G, X@$P_1^*(G, X)$} is defined in the same way with $K$ finite instead of compact.

\begin{thm}\label{anker}
Let $G$ be a locally compact topological group. For a compact continuous $G$-space $X$, Properties $(P_1(G, X))$ and $(P_1^*(G, X))$ are equivalent.
\end{thm}

The proof is split into two parts. The first one is a proposition that applies to continuous Banach $G$-modules and gives a general construction to obtain $\epsilon$-invariant vectors under the action of a compact subset $K$ of $G$. The second one is the proof itself : a closer look at the construction made in the proposition will be taken in the case where the Banach $G$-module is $C(X, L^1G)$.

Our development will more or less fit to the detail level of~\cite{Anker}. In particular, the existence of some subsets of $G$ having some nice (however not miraculous) properties will be used without any proof.

\begin{prop}\label{anker:prop}
 Let $E$ be a continuous Banach $G$-module and assume that for every finite subset $F$ of $G$ and every $\epsilon >0$ there exists a $(F, \epsilon)$-invariant vector $v_{F, \epsilon}$ with $\|v\| =1$. For every $\psi$ in $S(G)$ with compact support the set $\{\psi * v_{F, \epsilon} : \epsilon >0, F \subseteq G \text{ finite}\}$ contains a $(K, \epsilon)$-invariant vector for every compact subset $K$ of $G$.
\end{prop}
\begin{proof}
 Let $P:=\{v_{F, \epsilon} : F \subseteq G \text{ finite and } \epsilon >0\}$ be a set of $(F, \epsilon)$-invariant vectors in $E$ of norm one. Let $K$ be a compact subset of $G$ and fix $\epsilon >0$. Finally, choose any $\psi$ in $S(G)$ with compact support. We will show that there is a $v$ in $P$ such that $\psi * v$ is $(K, \epsilon)$-invariant. Notice that the result will also be proved if there exists a $v$ in $P$ such that $\psi * v$ is $(K \cup \{1_G\}, \epsilon)$-invariant. It can therefore be assumed without loss of generality that $K$ contains $1_G$.

Set $\eta := \epsilon /10$. By Proposition~\ref{continue_sur_L1:prop}, there exists a neighborhood $V'$ of $1_G$ such that 
\begin{equation}\label{eq1:8.12.PM}
\|\delta_g * \psi - \psi\|_1 = \|g \cdot \psi - \psi\|_1 \leq \eta \quad (g \in V').
\end{equation}
If we define $V:=\bigcap_{g \in K} g V' \inv g$, then $V$ is still a neighborhood of $1_G$ and by definition of $V$, it has the property that
\begin{equation}\label{eq2:8.12.PM}
 \sup_{h \in \inv{g} V g} \|\delta_h * \psi - \psi\|_1 \leq \eta \quad (g \in K).
\end{equation}
Let $\alpha$ be a function in $S(G)$ whose support is contained in $V$. Since
$$
	\alpha * \delta_g * \psi (z) - \delta_g * \psi (z) = \int_G \alpha (t)
[( \delta_{tg} * \psi )(  z)  - ( \delta_g
* \psi)(z) ]\, d\lambda(t),
$$
it follows from equation~(\ref{eq2:8.12.PM}) that the following estimate holds :
\begin{eqnarray}\label{eq1:9.12}
\|\alpha * \delta_g * \psi - \delta_g * \psi \|_1 	
&\leq& 	\|\alpha\|_1 \sup_{h \in \inv{g} V g} \|\delta_h * \psi - \psi\|_1 \leq \eta \quad (g \in K).
\end{eqnarray}

Now, let $M'$ be a compact subset of $G$ such that the integral of $\psi$ on $M'^c$ is less or equal to $\eta$ and define $M :=KM'$. Since $gM'$ is contained in $M$ for every $g$ in $K$, the set $\inv{g} M$ contains $M'$ for every $g$ in $K$ and therefore
\begin{equation}\label{eq2:9.12}
\int_{G \backslash \inv{g} M} \psi \, d\lambda \leq \int_{G \backslash M'} \psi \, d\lambda \leq \eta, \quad (g \in K).
\end{equation}
There exists a compact neighborhood $W$ of $1_G$ such that $\|\alpha * \delta_g -\alpha\|_1 \leq \eta$ for every $g$ in $W$. Let $A$ be another neighborhood of $1_G$, but this time an open one, such that $A \inv{A}$ is contained in $W$. By compactness of the set $M$ defined above, there exist $m_1, \dots, m_n$ in $M$ such that
$$
M \subseteq Am_1 \cup \cdots \cup A m_n.
$$
If for every $i$ in $\{1, \dots, n\}$ we define $M_i := A m_i$, then $W$ contains all of the $M_i \inv{M_i}$'s. This shows that up to a basic modification of the $M_i$'s we can assume that for every $i$ and $j$ in $\{1, \cdots, n\}$ they verify
$$
M_i \in \mathcal{B}(G), \quad M = \bigcup_{k=1}^n M_k, \quad M_i \cap M_j = \emptyset \text{ for every } i\neq j \quad \text{ and } \quad M_i \inv{M_i} \subseteq W.
$$
Moreover, up to letting go some of the $M_i$'s, we can assume that they are nonempty. If for every $i$ in $\{1, \cdots, n\}$ a $g_i$ is chosen in $M_i$, then we have
\begin{equation}\label{eq1:10.12}
 \|\alpha * \delta_t - \alpha * \delta_{g_i} \|_1 = \| (\alpha * \delta_{t \inv{g_i}}- \alpha) * \delta_{g_i}\|_1 = \|\alpha * \delta_{t \inv{g_i}}- \alpha\|_1 \leq \eta \quad (t \in M_i),
\end{equation}
by definition of $W$ and the fact that $M_i \inv{M_i}$ is contained in $W$.

 Define $F :=\{g_1, \dots, g_n\}$ and let $v$ be a $(F, \eta)$-invariant vector of norm one. We claim that $\psi * v$ is $(K, \epsilon)$-invariant. This is true because the right-hand side terms of the inequality
\begin{multline}\label{eq:theone.10.12}
 \|g \cdot (\psi * v) - \psi * v\| \leq \|\delta_g * \psi * v - \alpha * \delta_g * \psi * v\| + \|\alpha * \delta_g * \psi * v - \alpha * v\|\\
+\|\alpha  * v - \alpha *\psi * v\| + \| \alpha *\psi * v - \psi * v\|
\end{multline}
will shortly be seen to be sufficiently small whenever $g$ belongs to $K$. First, notice that the last two terms in the right-hand side are nothing but the first two ones with $g$ equal to $1_G$ and therefore, $1_G$ belonging to $K$, it is enough to obtain a suitable bound for the first two terms.

Using part (b) of Theorem~\ref{basic_properties_integral:thm} and equation~\eqref{eq1:9.12}, we have
$$
\|\delta_g * \psi * v - \alpha * \delta_g * \psi * v\| \leq \|\delta_g * \psi  - \alpha * \delta_g * \psi \|_1 \|v\| \leq \eta, \quad (g \in K)
$$
and we are done with the first term. Let us tackle the second one. For simplicity, write $\varphi$ instead of $\delta_g * \psi$. Since the map $(s, t) \mapsto \alpha(s) \varphi(\inv{s} t) (t \cdot v)$ from $G \times G$ to $V$ has compact support, Fubini's theorem for the Bochner integral can be used to obtain
$$
\alpha * \varphi * v = \int_G \varphi(t) [\alpha * (t \cdot v)] \, d\lambda(t).
$$
This equality, together with the fact that $\varphi$ belongs to $S(G)$, leads to the inequality
$$
\|\alpha * \varphi * v - \alpha * v\|	 \leq 	\int_G |\varphi(t)| \, \|\alpha* (t \cdot v) - \alpha * v\| \, d\lambda(t),
$$
whose right-hand side is less or equal to
\begin{multline*}
\int_{G \backslash M} |\varphi(t)| \, \|\alpha* (t \cdot v) - \alpha * v\| \, d\lambda(t) + \sum_{i=1}^n \int_{M_i} |\varphi(t)| \, \|\alpha* (t \cdot v) - \alpha * (g_i \cdot v)\| \, d\lambda(t) \\
+  \sum_{i=1}^n \int_{M_i} |\varphi(t)| \, \|\alpha * (g_i \cdot v) - \alpha * v\| \, d\lambda(t).
\end{multline*}
We claim that this cumbersome expression is in fact less or equal to $4 \eta$. To see it, keep in mind that $v$ is $(F, \eta)$-invariant and use the estimates
$$
\|\alpha* (t \cdot v) - \alpha * v\| \leq \|\alpha\| (\|t \cdot v\| + \|v\|) \leq 2
\quad \text{ and }  \quad \|\alpha * (t \cdot v) - \alpha * (g_i \cdot v)\| \leq \|\alpha * \delta_t - \alpha * \delta_{g_i}\|
$$
together with equations~\eqref{eq2:9.12} and~\eqref{eq1:10.12}.

Gathering all the estimates obtained by now, we have $\|g \cdot (\psi * v) - \psi * v\| \leq 10 \eta = \epsilon$ for every $ g$ in $K$.
\end{proof}

We now turn to the proof of Theorem~\ref{anker}. It will essentially consist in making some simple observations because all the arduous work was in the above proposition.

\begin{proof}[Proof of Theorem~\ref{anker}]
 First of all, it is clear that $(P_1^*(G, X))$ holds, whenever the group $G$ satisfies property $(P_1(G, X))$.

In order to prove the other implication, assume that Property $(P_1^*(G, X))$ is verified and let $\{f_{F, \epsilon}  : F \subseteq G \text{ finite and } \epsilon > 0\}$ be a set of $(F, \epsilon)$-invariant maps in $C(X, L^1G)$ ranging in $S(G)$. In particular, all of the $f_{F, \epsilon}$'s are of norm 1 and so the continuous Banach $G$-module $C(X, L^1G)$ (see Corollary~\ref{diag_action_L^1_continuous:cor}) satisfies the hypothesis of Proposition~\ref{anker:prop}.

Thus, by Proposition~\ref{anker:prop}, we know that if $\psi$ is any function in $S(G)$ with compact support, then the set
$$
\{\psi * f_{F, \epsilon} : F \subseteq G \text{ finite and } \epsilon > 0\}
$$
contains a $(K, \epsilon)$-invariant vector for every compact subset $K$ of $G$. Let $Z$ be a compact subset of $G$ with positive measure and define a function $\psi$ in $S(G)$ by $\psi : = \indicc{Z} / \lambda(Z)$.

Now, some observations must be made. Assume that $f$ belongs to $C(X, L^1G)$ and ranges in $S(G)$. By Theorem~\ref{commute_lineaires} applied to the evaluation-at-$x$ linear map, we have
\begin{equation}\label{eq3:10.12}
 (\psi * f)(x)= \int_Z t \cdot f(\inv{t}x) \, d\lambda(t) = \lim_{n \to \infty}\sum_{i=1}^{b(n)} \lambda (A_{i, n}) (t_{i, n} \cdot f(\inv{t_{i, n}} x))
\end{equation}
where for every integer $n$ the set $\{A_{1, n}, \dots, A_{b(n), n}\}$ is a partition of $Z$ into measurable subsets of $G$. Since the $L^1$-norm is additive on $S(G)$, it follows from~\eqref{eq3:10.12} that
$$
\|(\psi * f)(x)\|_1 = \lim_{n \to \infty} \left\|\sum_{i=1}^{b(n)} \lambda (A_{i, n}) (t_{i, n} \cdot f(\inv{t_{i, n}} x))\right\|_1 =\lambda(Z),
$$
for every $x$ in $X$. Moreover, the maps ranging in $S(G)$ are contained in the positive cone of the Banach lattice $C(X, L^1G)$ (see Example~\ref{C(X, L1)_lattice:exple}). The latter being closed, equation~\eqref{eq3:10.12} implies that $\psi * f$ is positive.

These observations imply that if the normalized maps $\tilde{f}_{F, \epsilon}$ are defined by
$$
\tilde{f}_{F, \epsilon}(x) := \frac{(\psi * f_{F, \epsilon})(x)}{\lambda(Z)}, \quad (x \in X),
$$
then the set $\{\tilde{f}_{F, \epsilon} : F \subseteq G \text{ finite and } \epsilon >0\}$ contains a $(K, \epsilon)$-invariant map for every compact $K$ of $G$ and every $\epsilon >0$. Moreover, every map in this set ranges in $S(G)$. This leads to the end of the proof.
\end{proof}

\chapter{Topological amenability}\label{result:Ch}
  Let $G$ be a discrete group acting by homeomorphisms on a compact space $X$. In~\cite{AVENIR}, it is proved (among other things) that the following statements are equivalent.
\begin{itemize}
 \item[(a)]The action of $G$ on $X$ is topologically amenable.
  \item[(b)]Every dual $(G, X)$-module of type C is a relatively injective Banach $G$-module.
  \item[(c)]There is a $G$-invariant element in $C(X, l^1G)^{**}$ summing to $\indicc{X}$.
  \item[(d)]There is a norm one positive $G$-invariant element in $C(X, l^1G)^{**}$ summing to $\indicc{X}$.
\end{itemize}

The main goal of this master thesis was to generalize the result to general locally compact topologies on $G$. In other words and more precisely, we wanted to establish a similar result under the hypothesis that $G$ is a locally compact group acting continuously on a compact space $X$. It will be proved in this chapter that in the described locally compact setting, the statements above remain equivalent whenever we replace $l^1G$ by $L^1G$.

\section{Amenable transformation groups}

This section is almost a copy of Section 2 of~\cite{Delaroche}. Transformation groups and amenable transformation groups are defined and several equivalent definitions we will work with are given. We recall that all the following topological spaces have the Hausdorff property.

\begin{defn}
 A \emph{transformation group}\index{transformation group} is a left $G$-space $X$, where $X$ is a locally compact space and $G$ a locally compact group acting continuously from the left on $X$.
\end{defn}

 We will denote by $\mathrm{Prob}(G)$\index{Prob G@$\mathrm{Prob}(G)$} the set of probability measures on the Borel subsets of $G$. Topologically, the set $\mathrm{Prob}(G)$ will be seen as a subset of $C_0 (G)^*$\index{C 0 G@$C_0(G)$} and be endowed with the induced weak-* topology. In other words, a probability measure is seen as continuous linear functional on the space of continuous functions on $G$ vanishing at infinity. In this setting, $G$ acts on $\mathrm{Prob}(G)$ through its action on $C_0 (G)$ : for $g$ in $G$ and $m$ in $\mathrm{Prob}(G)$, we define $gm(f) :=m(\inv{g}f)$ for every $f$ in $C_0(G)$.

\begin{defn}
 A transformation group $(X, G)$ is \emph{amenable}\index{amenable transformation group} if there exists a net $\{m_n\}_{n \in N}$ of continuous maps $x \mapsto m_n^x$ from $X$ to $\mathrm{Prob}(G)$ such that
$$
\lim_n \|g m_n^x - m_n^{gx}\|_1 = 0
$$
uniformly on compact subsets of $X \times G$.
\end{defn}

The following proposition and proof correspond to Proposition 2.2 in~\cite{Delaroche}. 

\begin{prop}\label{equiv_def_amenability}
 The following conditions are equivalent.
\begin{itemize}
 \item[(a)]$(X, G)$ is an amenable transformation group.
  \item[(b)]There exists a net $\{f_n\}_{n \in N}$ of nonnegative continuous functions on $X \times G$ such that
	    \begin{itemize}
		\item[(1)]for every $n$ in $N$ and $x$ in $X$, $\int_G f_n(x, t) \, d\lambda(t) = 1$;
		\item[(2)]$\lim_n \int_G |f_n(gx, gt) - f_n(x, t)| \, d\lambda(t)=0$ uniformly on compact subsets of $X \times G$.
	     \end{itemize}
  \item[(c)]There exists a net $\{f_n\}_{n \in N}$ in $C_c(X \times G)^+$ such that
	    \begin{itemize}
		\item[(1)]$ \lim_n \int_G f_n(x, t) \, d\lambda(t) = 1$ uniformly on compact subsets of $X$;
		\item[(2)]$\lim_n \int_G |f_n(gx, gt) - f_n(x, t)| \, d\lambda(t)=0$ uniformly on compact subsets of $X \times G$.
	     \end{itemize}
\end{itemize}
\end{prop}
\begin{proof}
 (a) $\Rightarrow$ (b). Let $f$ belong to $C_c(G)^+$ be such that $\int_G f(t) \, d\lambda(t) = 1$ and set
$$
f_n(x, g) = \int_G f(\inv{t} g) \, dm_n^x(t).
$$
By the Fubini theorem, we get $\int_G f_n(x, s) \, d\lambda(s) = 1$ for every $x$ in $X$ and $n$ in $N$. Moreover, for $(x, g)$ in $X\times G$, the invariance of the left Haar measure implies that
\begin{eqnarray*}
 \|\inv{g}f_n(gx, \cdot) - f_n(x, \cdot)\|_1&=& \int_G \left| \int_G f(\inv{u} gt) \, dm_n^{gx}(u) - \int_G f(\inv{u} t) \, dm_n^x(u)\right| \, d\lambda(t)\\
&\leq& \int_G \int_G f(\inv{u} g t) \, d|m_n^{gx}-gm_n^x|(u) \, d\lambda(t),
\end{eqnarray*}
where $|m_n^{gx}-gm_n^x|$ is the total variation of $m_n^{gx}-gm_n^x$. Using again the Fubini theorem, we obtain the majoration
$$
 \|\inv{g}f_n(gx, \cdot) - f_n(x, \cdot)\|_1 \leq \int_G \int_G f(\inv{u} gt) \, d\lambda(t) \, d|m_n^{gx}-gm_n^x|(u) \leq \|m_n^{gx}-gm_n^x\|_1.
$$
This tends to zero uniformly on compact subsets of $X \times G$.

(b) $ \Rightarrow$ (c). An easy approximation argument allows us to replace the net $\{f_n\}_{n \in N}$ by a net in $C_c(X \times G)^+$ satisfying conditions (1) and (2) of (c).

(c) $ \Rightarrow$ (b). Let $\{f_n\}_{n \in N}$ be as in (c) and choose $f$ in $C_c(G)^+$ as above. Let us define $\{f_{n, i}\}_{(n, i) \in (N \times \N^*) }$ by
$$
f_{n, i}(x, g)=\frac{f_n(x, g) + \frac{1}{i} f(g)}{\int_G f_n(x, t) \, d\lambda(t) + \frac{1}{i}}.
$$
Then $\{f_{n, i}\}_{(n, i) \in (N \times \N^*)}$ satisfies conditions (1) and (2) of (b).

(b) $ \Rightarrow$ (a) is obvious : given $\{f_n\}_{n \in N}$ as in (b), we define $m_n^x$ to be the probability measure with density $f_n(x, \cdot)$ with respect to the left Haar measure. Then $\{m_n\}_{n \in N}$ satisfies the conditions in the definition of amenable transformation group.
\end{proof}

\section{Auxiliary results}
In this section, some results that will be useful in the subsequent proofs are given. It is to be taken as a miscellanea since the results have no particular connexion between each other.

\begin{lem}\label{auxiliaryLemma1}
 Let $G$ be a locally compact group, $X$ a compact space and $E$ a Banach $(G, X)$ module of type C. If $v$ is a map in $C(G, E)$, such that $\|v(g)\| \leq C$ for every $g$ in $G$ with $C$ in $\R_+$, then
$$
\left\| \int_G f(t) \cdot v(t) \, d\lambda(t)\right\| \leq  \left\| \int_G |f(t)| \, d\lambda(t)\right\| C
$$
for every $f$ in $C_c(G, C(X))$.
\end{lem}

Notice that this lemma is nothing but the integral version of the type C condition for Banach $(G, X)$-modules.

\begin{proof}
Let $v$ and $C$ be as in the statement of the lemma and take $f$ in $C_c(G, C(X))$. First of all, notice that both integrals above exist by Remark~\ref{integrale_C_c(G):rmk}. Define a map $\Gamma : G \rightarrow E$ by declaring that $\Gamma(g) = f(g) \cdot v(g)$ for every $g$ in $G$. Since $f$ has compact support, so has $\Gamma$ and let $K$ be its compact support. Let $M$ be defined by
$$
M:= \sup_{g \in G}\|f(g)\|_{C(X)} + C.
$$
The continuity of $f$ and $v$ implies that for every integer $n$ and every $g$ in $G$ some open subset $O_g^n$ of $G$ such that
$$
\|f(s) - f(g)\|_{C(X)} \leq \frac{1}{2 M n} \quad \text{and} \quad \|v(s)-v(g)\|_E \leq \frac{1}{2 M n} \quad (s \in O_g^n)
$$
can be found. Notice that this implies that
$$
\| \Gamma(s) - \Gamma (g)\| \leq \|f(s)\| \|v(s) - v(g)\| + \|f(s) - f(g)\| \|v(g)\| \leq \frac{1}{n} \quad (s \in O^n_g).
$$
Since $K$ is compact, for every integer $n$ we can find $g^n_1, \dots, g_{m_n}^n$ in $G$ such that $K \subseteq O^n_{g^n_1} \cup \cdots \cup O^n_{g^n_{m_n}}$. For every integer $n$ define Borel sets by $A^n_1 := O^n_{g^n_1}$,
$$
\quad A_i^n : = O^n_{g^n_i} \backslash \bigcup_{1 \leq j < i} A^n_j \quad \text{ and simple maps by } \quad s_n := \sum_{i=1}^{m_n} \indicc{A^n_i} f(g^n_i) v(g_i^n),
$$
where the index $i$ ranges in $\{2, \dots, m_n\}$. An argument similar to the one used in Remark~\ref{integrale_C_c(G):rmk} shows that
\begin{equation}\label{eq1:14.12}
\int_G f(t) \cdot v(t) \, d\lambda(t) = \lim_{n \to \infty}\sum_{i=1}^{m_n} \lambda(A^n_i) f(g^n_i) \cdot v(g_i^n)
\end{equation}
and
$$
\int_G |f(t)| \, d\lambda(t) = \lim_{n \to \infty}\sum_{i=1}^{m_n} \lambda(A^n_i) |f(g^n_i)|.
$$
Thus, taking the norm on both sides of equation~\eqref{eq1:14.12} and using the fact that $E$ is of type C, the proof is complete.
\end{proof}

%
%
%
%

Before stating our next result, let us give a brief description of the context in which it will be used. This may render this section somewhat more digest and throw light on the later use of the result in question. Assume we have a short exact sequence of $G$-morphisms of Banach $G$-modules
$$
0 \longrightarrow A \overset{\alpha}{\longrightarrow} B \overset{\beta}{\longrightarrow} C \longrightarrow 0
$$
and that we are looking for a $G$-invariant element in $B$. Assume furthermore that we know that there are $G$-invariant elements in $C$. We wish we could write the short exact sequence
$$
0 \longrightarrow A^G \overset{\alpha}{\longrightarrow} B^G \overset{\beta}{\longrightarrow} C^G \longrightarrow 0,
$$
where the superscript $^G$ denotes the $G$-invariant elements. Sadly, this not true in general and the sequence stops at $B^G$. However, there is a simple condition for this to hold, namely the existence of a $G$-morphism $\sigma : C \rightarrow B$ such that $\beta \circ \sigma = Id_C$. The next results give conditions equivalent to the existence of such a $G$-morphism. They are both almost exact copies of Proposition 4.2.1 and Corollary 4.2.6 of~\cite{These_Monod}.

\begin{prop}\label{prepa:prop}
 Let $\eta : A \rightarrow B$ be a $G$-morphism of Banach $G$-modules. If both subspaces $\mathrm{Ker}(\eta)$ and $\mathrm{Im}(\eta)$ admit $G$-invariant complements, then there exists a $G$-morphism $\sigma : B \rightarrow A$ such that $\eta \sigma \eta = \eta$.
\end{prop}
\begin{proof}
 Assume that both subspaces $\mathrm{Ker}(\eta)$ and $\mathrm{Im}(\eta)$ admit invariant complements. The fact that these subspaces are complemented implies that there exists surjective idempotent morphisms $p : A \rightarrow \mathrm{Ker}(\eta)$ and $q :B \rightarrow \mathrm{Im}(\eta)$ ; and the fact that these complements are $G$-invariant implies that both $p$ and $q$ are $G$-morphisms.

Let $\pi : A \rightarrow A / \mathrm{Ker}(\eta)$ be the canonical projection and $\bar{\eta} : A/\mathrm{Ker}(\eta) \rightarrow \mathrm{Im}(\eta)$ the induced morphism. Since $\mathrm{Ker}(\eta)$ is $G$-invariant, the action of $G$ on $A$ induces a well defined action on $A / \mathrm{Ker}(\eta)$ with respect to which $\bar{\eta}$ is equivariant. Since $\mathrm{Im}(\eta)$ is closed (it is complemented), $\bar{\eta}$ is an isomorphism of topological vector spaces and hence admits an inverse morphism $\inv{\bar{\eta}}$. The map $Id - p$ vanishes on $\mathrm{Ker}(\eta)$ and therefore induces a $G$-morphism $\overline{Id - p} : A/\mathrm{Ker}(\eta) \rightarrow A$. The map $\sigma : B \rightarrow A$ is defined by
$$
\sigma := (\overline{Id - p})\inv{\bar{\eta}}q.
$$
At this step, the relations $q \eta = \eta$, $\inv{\bar{\eta}} \eta = \pi$ and $(\overline{Id - p})\pi = Id - p $ can be used in order to verify that $\eta \sigma \eta = \eta$. Since all the maps through which $\sigma$ is defined are equivariant, so is $\sigma$.
\end{proof}

\begin{cor}\label{G_split}
Let
$$
0 \longrightarrow A \overset{\alpha}{\longrightarrow} B \overset{\beta}{\longrightarrow} C \longrightarrow 0
$$
be a short exact sequence of $G$-morphisms of Banach $G$-modules. The following assertions are equivalent.
\begin{itemize}
 \item[(a)]There exists a left inverse $G$-morphism for $\alpha$.
  \item[(b)]There exists a right inverse $G$-morphism for $\beta$.
  \item[(c)]$\mathrm{Im}(\alpha) = \mathrm{Ker}(\beta)$ admits a $G$-invariant complement in $B$.
\end{itemize}
\end{cor}
\begin{proof}
First, we show that both (a) and (b) imply (c). If $\sigma : B \rightarrow A$ is a left inverse $G$-morphism for $\alpha$, then $\mathrm{Im}(\alpha) = \mathrm{Im}(\alpha \sigma)$. Since $\alpha \sigma$ is idempotent, we obtain that $\mathrm{Im}(\alpha)$ is complemented. Its complement being $\mathrm{Ker}(\alpha \sigma)$, it is clear that the complement in question is $G$-invariant. If $\gamma : C \rightarrow B$ is a right inverse $G$-morphism for $\beta$, a similar argument shows that $\mathrm{Ker}(\beta) = \mathrm{Im}(Id- \gamma \beta)$. Again, using the fact that $Id- \gamma \beta$ is idempotent and equivariant, it is easy to see that $\mathrm{Ker}(\beta)$ admits a $G$-invariant complement.

To prove that (c) implies (a) and (b), it suffices to apply Proposition~\ref{prepa:prop} to the cases $\eta = \alpha$ and $\eta = \beta$,  using the injectivity of $\alpha$ and the surjectivity of $\beta$.
\end{proof}
%
%
%
%
For the next result, some notations must be established. As usual, let $G$ be a locally compact group and $X$ a compact space. We will write $I$\index{I@$I$ map} (for integral) the linear map from $C(X, L^1G)$ to $C(X)$ whose images are defined on $X$ by
\begin{equation}\label{operatorI}
If(x) := \int_G f(x)(t) \, d\lambda(t).
\end{equation}
Notice that $I$ is continuous and that $\| I \|  \leq 1$. Define a subspace $H$ of $C(X, L^1G)$ by
$$
H:=\{f \in C(X, L^1G) : I f \in \R \indicc{X}\}.
$$
Since $H$ is the pre-image of the closed subset $\R \indicc{X}$ under the continuous map $I$, the subspace $H$ is closed and it is therefore a Banach space. Using dual maps, we can also define a subspace $\tilde{H}$ of the double dual of $C(X, L^1G)$ by
$$
\tilde{H}:=\{f \in C(X, L^1G)^{**} : I^{**} f \in \R j(\indicc{X})\},
$$
where $j$ denotes the canonical embedding of $C(X)$ into its double dual. Elements that are mapped to $\indicc{X}$ or $j(\indicc{X})$ by $I$ or its double dual map are said to \emph{sum up to one}\index{sum up to one}. Finally, let $H^\perp$ be the annihilator of $H$ in $C(X, L^1G)^*$.

\begin{prop}\label{auxiliaryProp1}
 With the above notations, if $\iota : H \rightarrow C(X, L^1G)$ denotes the inclusion, then $\tilde{H} = \iota^{**}(H^{**})$.
\end{prop}

Proving this proposition becomes really easy with the help of the two following lemmas.

\begin{lem}\label{lem1:14.12}
 The orthogonal complement of $H$ is contained in $I^* (C(X)^*)$.
\end{lem}
\begin{proof}
Choose $\psi$ in $L^1G$ with $\int_G \psi \, d\lambda = 1$ and define a continuous linear map $T : C(X) \rightarrow C(X, L^1G)$ by declaring that $Tk(x) = k(x) \psi$ for every $x$ in $X$. Now, observe that if $h^\perp$ belongs to $H^\perp$, then for every $f$ in $C(X, L^1G)$ the value of $h^\perp(f)$ only depends on $If$. In fact, if $f$ and $f'$ are maps in $C(X, L^1G)$ such that $If = If'$, then $I(f-f')$ belongs to $\R \indicc{X}$ and therefore $h^\perp(f-f')=0$.

Let $h^\perp$ belong to $H^\perp$ and let us show that $h^\perp = I^*k^*$ for some $k^*$ in $C(X)^*$. Define $k^*$ in $C(X)^*$ by
$$
k^*(k) := h^\perp(Tk) \quad (k \in C(X)).
$$
Since $I(T(If)) = If$ for every $f$ in $C(X, L^1G)$, we have $h^\perp(T(If)) = h^\perp(f)$. It follows that $h^\perp = I^*k^*$.
\end{proof}

\begin{lem}\label{lem2:14.12}
 If $\mu$ belongs to $\tilde{H}$, then $\mu(h^\perp) = 0$ for every $h^\perp$ in $H^\perp$.
\end{lem}
\begin{proof}
 Let $h^\perp$ belong to $H^\perp$. By Lemma~\ref{lem1:14.12}, we know that $h^\perp = I^* k^*$ for some $k^*$ in $C(X)^*$. Let $\psi$ be an element of $L^1G$ such that $\int_G \psi \, d\lambda = 1$ and see it as a constant map in $C(X, L^1G)$. If $\mu$ belongs to $\tilde{H}$, then for some $r$ in $\R$, we have
\begin{equation*}
\mu(h^\perp) = I^{**}\mu (k^*) = r \, k^* ( \indicc{X})  = r \, k^* (I\psi) = r \, I^*k^* (\psi) = r \, h^\perp (\psi) = 0. \qedhere
\end{equation*}
\end{proof}

\begin{proof}[Proof of Proposition~\ref{auxiliaryProp1}]
First, we show that $\iota^{**}(H^{**})$ is contained in $\tilde{H}$. Let $\mu$ belong to $H^{**}$. It must be shown that $I^{**}(\iota^{**}\mu)$ belongs to $\R j(\indicc{X})$, where $j$ is the canonical embedding of $C(X)$ into its double dual. Since there exists a net in $H$ that weak-* converges to $\mu$, it follows that $I^{**}(\iota^{**}\mu)$ is the weak-* limit of a net whose elements are of the form $r_n j(\indicc{X})$, where the $r_n$'s belong to $\R$.

We can suppose that $I^{**}(\iota^{**} \mu) \neq 0$. Otherwise there is nothing to show. Let $k^*$ in $C(X)^*$ be such that $I^{**}(\iota^{**} \mu)(k^*)= r \neq 0$. We have
$$
0 \neq r = I^{**}(\iota^{**} \mu)(k^*) = \lim_{n} r_n k^*(\indicc{X}).
$$
This implies that $I^{**}(\iota^{**} \mu) = [r/k^*(\indicc{X})] j(\indicc{X})$.

To show that $\tilde{H}$ is contained in $\iota^{**}(H^{**})$, let $\xi$ belong to $\tilde{H}$. We have to find a $\mu$ in $H^{**}$ such that $\xi = \iota^{**} \mu$. For every $h^*$ in $H^*$, define
$
\mu (h^*) = \xi (\overline{h^*}),
$
where $\overline{h^*}$ is any extension of $h^*$ to $C(X, L^1G)$ such that $\|\overline{h^*}\|=\|h^*\|$. By Lemma~\ref{lem2:14.12}, this gives rise to a well defined bounded linear map. In other words, $\mu$ belongs to $H^{**}$. Moreover, for every $f^*$ in $C(X, L^1G)^*$, we have
$$
\iota^{**}( \mu) (f^*) = \mu (\iota^* f^* )= \xi (\overline{\iota^* f^*}) = \xi (f^*).
$$
This shows that $\xi$ belongs to $\iota^{**}(H^{**})$.
\end{proof}

%
%
%
%

The following is a generalization of a well known result about weak convergence.

\begin{prop}\label{auxiliaryProp2}
 Let $E$ be a Banach space, $G$ be any set and for every $g$ in $G$, let $F_g$ be a continuous linear map from $E$ to $E$. If $\{v_n\}_{n \in N}$ is a net in $E$ such that $\{F_g(v_n)\}_{n \in N}$ converges weakly to zero for every $g$, then there exists a net of convex combinations of the $v_n$'s, whose elements are still denoted by $v_n$, such that $\{F_g(v_n)\}_{n \in N}$ converges strongly to zero for every $g$.
\end{prop}
\begin{proof}
 Let $\Pi$ be the product over $G$ of $E$, that is to say $\Pi = \Pi_{g \in G} E$, and endow $\Pi$ with the product topology. By considering pre-images of a neighborhood of $0$ in $\R$, notice that the continuous linear functionals on $\Pi$ are zero on all but a finite number of components. Based on this observation, one can see that if $f$ belongs to $\Pi^*$, then it is of the form
$
f_{g_1} + \cdots + f_{g_m},
$
where the $f_{g_i}$'s belong to $E^*$. In other words,
$$
f(\{v_g\}_{g \in G}) = f_{g_1}(v_{g_1}) + \cdots + f_{g_m}(v_{g_m}) \quad (\{v_g\}_{g \in G} \in \Pi).
$$
It follows that if $\{v_n\}_{n \in N}$ is a net as in the statement, then the net in $\Pi$ constituted by the $F_g(v_n)$'s converges weakly to zero in $\Pi$.

This implies that there is a net of convex combinations of the $F_g(v_n)$'s that converges strongly to zero in $\Pi$ (this is a consequence of the Hahn-Banach theorem, which implies that the weak-closure and the norm-closure of a convex set coincide). Using the linearity of the $F_g$'s to see a convex combination of $F_g(v_n)$'s as merely $F_g$ applied to a convex combinations of the $v_n$'s, that is to say
$$
\lambda_1 F_g(v_1) + \cdots + \lambda_n F_g (v_n) = F_g (\lambda_1 v_1 + \cdots + \lambda_n v_n),
$$
we are done.
\end{proof}

\section{A theorem about amenable transformation groups}\label{result:Sect}
It is finally the time to state and proof the result mentioned in the opening of this chapter. This whole section is devoted to its proof. We recall that all the following topological spaces have the Hausdorff property.

\begin{thm}\label{main_result}
 Let $(G, X)$ be a transformation group, where $X$ is compact and $G$ locally compact. The following are equivalent.
\begin{itemize}
 \item[(a)]$(G, X)$ is an amenable transformation group.
  \item[(b)]Every dual $(G, X)$-module of type C is a relatively injective Banach $G$-module.
  \item[(c)]There is $G$-invariant element in $C(X, L^1G)^{**}$ summing to $\indicc{X}$ in $C(X)^{**}$.
  \item[(d)]There is a norm one positive $G$-invariant element in $C(X, L^1G)^{**}$ summing to $\indicc{X}$ in $C(X)^{**}$.
\end{itemize}
\end{thm}

The plan of the proof is the following. It will be shown that (a) $\Rightarrow$ (b)$\Rightarrow $ (c) $\Rightarrow $ (a) and then that (a) $\Rightarrow $ (d). The missing logical step is (d) $\Rightarrow $ (c), but this implication calls for no proof.

\subsubsection*{(a) implies (b)}
Let $E$ be a dual $(G, X)$-module of type C and consider the extension problem of Banach $G$-modules
$$
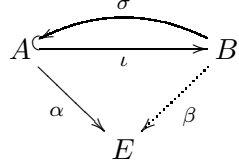
\xymatrix{ A \ar@{^{(}->}[rr]_\iota \ar[rd]_\alpha && B \ar@{.>}[ld]^\beta \ar@/_1pc/[ll]_\sigma \\ &E}
$$
where $\sigma$ is a norm one projection. For every $f$ in $C_c(X \times G)$, we can define a map $\beta_f : B \rightarrow E$ by
$$
\beta_f(b) := \int_G f( \cdot, t) \cdot t \cdot \alpha (\sigma(\inv{t} b)) \, d\lambda(t),
$$
where $f(\cdot, t) \cdot$ denotes the action of the map $f(\cdot, t)$ in $C(X)$ on the vector $t \cdot \alpha (\sigma(\inv{t} b))$ in $E$. Using the equivariance of $\alpha$, the integrand can be written as $f( \cdot, t)   \cdot \alpha (t \sigma(\inv{t} b))$ and therefore $\beta_f(b)$ is well defined as the integral of a continuous map with compact support. Moreover $\beta_f$ is linear.

Since $E$ is a dual space, $\mathcal{L}(B, E)$ is isometrically isomorphic to $(B \hat{\otimes} V)^*$ (projective tensor product), where $V^* = E$. Under that isomorphism $\beta(b \otimes v) = \beta b (v)$ for every $b$ in $B$ and every $v$ in $V$. This implies that if a net $\beta_n$ in $\mathcal{L}(B, E)$ weak-* converges to some $\beta$, then $\beta_n b$ weak-* converges to $\beta b$ in $E$ for every $b$ in $B$.

Let $\{f_n\}_{n \in N}$ be a net in $C_c(X, G)^+$ as in part (c) of Proposition~\ref{equiv_def_amenability} and write $\beta_n$ for $\beta_{f_n}$. The properties of this net and Lemma~\ref{auxiliaryLemma1} imply that for $n$ large enough, $\beta_n$ is bounded with norm less or equal to $ 2 \|\alpha\|$. Let $\beta$ be any weak-* accumulation point of the $\beta_n$'s. We claim that $\|\beta\| \leq \|\alpha\|$, $\beta \circ \iota = \alpha$ and that $\beta$ is equivariant.

The claim about $\| \beta \|$ follows from the weak-* convergence of the $\beta_n b$'s towards $\beta b$ for every $b$, the fact that
$$
|\beta_n b (v)| \leq \left\| \int_G f_n(\cdot, t) \, d\lambda(t)\right\|_{C(X)} \|\alpha\| \|\sigma\| \|b\| \|v\|,
$$
which is true because of Lemma~\ref{auxiliaryLemma1}, and the properties of the net $\{f_n\}_{n \in N}$. That $\beta \circ \iota = \alpha$ follows from the convergence in $C(X)$ of $\int_G f_n (\cdot, t) \, d\lambda(t)$ towards $\indicc{X}$ and the fact that the Bochner integral commutes with continuous linear maps. In fact, using the equivariance of $\alpha$, we have
$$
\beta_n \iota(a) (v) = \left( \int_G f_n(\cdot, t) \alpha(a)\, d\lambda(t) \right) (v) = \left[  \left( \int_G f_n(\cdot, t)\, d\lambda(t) \right) \cdot \alpha(a) \right] (v),
$$
for every $a$ in $A$ and every $v$ in $V$. Since the left-hand side of this expression converges to $\beta \iota(a)(v)$ and the right-hand side to $\alpha(a)(v)$, we see that $\beta \circ \iota = \alpha$. To prove that $\beta$ is equivariant, use the triangle inequality to obtain that $|(g \cdot \beta b)(v) - \beta(gb)(v)| $ is less or equal to
$$
 |(g \cdot \beta b)(v) - (g \cdot \beta_n b)(v)|+|(g \cdot \beta_n b)(v) - \beta_n (gb)(v)|+|\beta_n (gb)(v) - \beta(gb)(v)|.
$$
The first and the last term converge to zero because $\beta_n b$ weak-* converges to $\beta b$ for every $b$ in $B$. The middle one is bounded from above by $\|g \cdot \beta_n b - \beta_n (gb) \| \|v\|$ and this expression is in turn bounded from above by
\begin{equation}\label{eq1:16.12}
2 \left\| \int_G |f_n(\cdot, t) - \inv{g} \cdot f_n( \cdot, gt)| \, d\lambda(t)\right\|_{C(X)} \|\alpha\|\|\sigma\| \|b\| \|v\|.
\end{equation}
To see this, use equation~\eqref{C(X)compatibility} and Lemma~\ref{auxiliaryLemma1}. The evaluation at $x$ in $X$ being a bounded linear map on $C(X)$, we have
$$
\left(\int_G |f_n(\cdot, t) - \inv{g} \cdot f_n( \cdot, gt)| \, d\lambda(t) \right) (x) = \int_G |f_n(x, t) -  f_n( gx, gt)| \, d\lambda(t)
$$
and so, thanks to the properties of the net $\{f_n\}_{n \in N}$, the expression in~\eqref{eq1:16.12} converges to zero.

\subsubsection*{(b) implies (c)}

Let $C(X, L^1_0G)$ be the kernel of the integration map $I$ defined by~\eqref{operatorI} and $\psi$ be any positive element in $L^1G$ summing to one. Define a map $\tau : C(X, L^1G) \rightarrow C(X, L^1_0G)$ by
$$
\tau (f) : = f - If \otimes \psi,
$$
where for every $\phi$ in $C(X)$ the map $\phi \otimes \psi$ in $C(X, L^1G)$ is given by  $x \mapsto \phi(x) \psi$. The map $\tau$ is a morphism of Banach spaces such that if $\iota$ denote the inclusion of $C(X, L^1_0G)$ in $C(X, L^1G)$, then $\tau \circ \iota = Id$. Double dualizing this equality (without adding superscripts to the morphisms for evident readability reasons), we are in the following situation.
$$
\xymatrix{ C(X, L^1_0G)^{**} \ar@{^{(}->}[rr]_\iota \ar[rd]_{Id} && C(X, L^1G)^{**} \ar@/_1pc/[ll]_\tau \\ & C(X, L^1_0G)^{**}}
$$
Reasoning as in Example~\ref{ExC(X)ModuldeTypeC}, we see that $C(X, L^1_0G)$ is a $(G, X)$-module of type C. Therefore, by Lemma~\ref{dualTypeC}, its double dual is a $(G, X)$-module of type C too. By hypothesis, it follows that $C(X, L^1_0G)^{**}$ is a relatively injective Banach $G$-module. According to Lemma~\ref{equivDefRelInj}, there exists a $G$-morphism $\sigma : C(X, L^1G)^{**} \rightarrow C(X, L^1_0G)^{**}$ that fits into the above diagram. This means that we have a short exact sequence of $G$-morphisms
$$
0 \longrightarrow C(X, L^1_0G)^{**} \overset{\iota}{\longrightarrow} C(X, L^1G)^{**} \overset{I^{**}}{\longrightarrow} C(X)^{**} \longrightarrow 0
$$
where the first $G$-morphism admits a left inverse $G$-morphism, namely $\sigma$. By Corollary~\ref{G_split}, $I^{**}$ admits a right $G$-morphism inverse. The existence of such a right inverse implies that the above short exact sequence remains exact when passing to the $G$-invariant elements. In particular, there is a $G$-invariant element $f$ in $C(X, L^1G)^{**}$ that is mapped to $j(\indicc{X})$ by $I^{**}$.

\subsubsection*{(c) implies (a)}

Let $S(G)$ be the subset of $L^1G$ defined by equation~\eqref{SG}. Suppose that there exists a net $\{f_n\}_{n \in N}$ in $C(X, L^1G)$ such that all the $f_n(X)$'s are contained in $S(G)$ and $\|g f_n -f_n \|$ converges to zero for every $g$ in $G$. This would imply that Property $(P_1^*(G, X))$ is verified. By Theorem~\ref{anker}, there would be another net $\{\tilde{f}_n\}_{n \in N}$ in $C(X, L^1G)$ such that all the $\tilde{f}_n(X)$'s are contained in $S(G)$ and $\|g \tilde{f}_n - \tilde{f} \|$ converges uniformly to zero on compacts subsets of $G$. Defining $m_n : X \rightarrow \mathrm{Prob}(G)$ by $m_n(x) = \tilde{f}_n(x) \lambda$, we obtain a net in $C(X, \mathrm{Prob}(G))$ witnessing the amenability of the transformation group.

Therefore, our task is to find a net $\{f_n\}_{n \in N}$ in $C(X, L^1G)$ such that all the $f_n(X)$'s are contained in $S(G)$ and $\|g f_n -f_n \|$ converges to zero for every $g$ in $G$. Let $\{f_n\}_{n \in N}$ be a norm bounded net in $C(X, L^1G)$ that weak-* converges to an invariant element $f$ like in part (c) of the theorem. By Proposition~\ref{auxiliaryProp1}, we can suppose that $If_n$ belongs to $\R \indicc{X}$ for every $n$ in $N$. Let $r_n$ be a real number such that $I f_n = r_n \indicc{X}$. Since the $f_n$'s weak-* converge to $f$, the $r_n$'s converge to 1. Therefore, considering normalizations of the form $f_n / If_n$, which still belong to $C(X, L^1G)$, we can moreover suppose that $If_n = 1$ for every $n$ in $N$. Using the triangle inequality, we can verify that these normalizations still weak-* converge to $f$. As a consequence, the net $\{g f_n - f_n\}_{n \in N}$ converges weakly in $C(X, L^1G)$ to 0 for every $g$ in $G$. According to Proposition~\ref{auxiliaryProp2}, up to passing to a net of convex combinations of the $f_n$'s, we can assume that the net strongly converges to zero.

Now, recall that $C(X, L^1G)$ is a Banach lattice and that in such spaces the inequality $||a|-|b|| \leq |a-b|$ always holds. Using the fact that the action of $G$ on $C(X, L^1G)$ commutes with the absolute value, we have
$$
\|g |f_n| - |f_n| \|  \leq \| \, | gf_n -f_n| \,\| = \| gf_n -f_n \|.
$$
We can therefore assume that we have a net of positive elements of $C(X, L^1G)$ such that the net of the $\|g f_n - f_n\|$ converges to zero for every $g$ in $G$. By replacing the $f_n$'s by their absolute values, the property that $If_n = \indicc{X}$ may have been lost. Instead of this equality, we now have $If_n \geq \indicc{X}$ for every $n$ in $N$. On the maps $k$ in $C(X, L^1G)$ such that $Ik \geq \indicc{X}$, the normalization $k \mapsto k / Ik$ happens to be equivariant and 2-Lipschitz. Therefore, the normalized net of the $f_n$'s still verifies that $\|g f_n -f_n\|$ converges to zero for every $g$ in $G$.

All those modifications of the initial net produce a new net fulfilling all the conditions stated three paragraphs above.

\subsubsection*{(a) implies (d)}

Let $\{f_n\}_{n \in N}$ be a net in $C_c(X, G)^+$ as in part (c) of Proposition~\ref{equiv_def_amenability} and see this net in $C(X, L^1G)$. Since for $n$ large enough this net is norm bounded, we can suppose that it weak-* converges to some $f$ in $C(X, L^1G)^{**}$. If $g$ belongs to $G$ and $\xi$ to $C(X, L^1G)^*$, we have
$$
|(g \cdot f)(\xi) - f(\xi)| = \lim_n |\xi(g f_n) - \xi(f_n)| \leq \lim_n \|\xi\| \|g \cdot f_n - f_n\|_{C(X, L^1G)}.
$$
Since the right-hand side converges to zero because of the properties of the net $\{f_n\}_{n \in N}$, this inequality shows that $f$ is $G$-invariant.

From the fact that the $f_n$'s are positive elements of $C(X, L^1G)$, it follows that $f$ is a positive element of $C(X, L^1G)^{**}$. In fact, if $\xi$ is a positive element in $C(X, L^1G)^*$, we have $\xi (f_n) \geq 0$ for every $n$ and so $f(\xi) = \lim_n \xi (f_n) \geq 0$. 

Let us show that $I^{**}f =j(\indicc{X})$, where $I$ is the linear map defined in~\eqref{operatorI} and $j$ the canonical embedding of $C(X)$ into its double dual. The properties of the net $\{f_n\}_{n \in N}$ imply that $If_n$ converges weakly to $\indicc{X}$ in $C(X)$. Thus, $j(If_n)$ weak-* converges to $j(\indicc{X})$ in $C(X)^{**}$ and so, letting the limits below be weak-* ones,
$$
I^{**} f = \lim_n I^{**}f_n =  \lim_n j (If_n) = j(\indicc{X}).
$$

Finally, we show that $\|f\|= 1$. Since $\|I\| \leq 1$, the fact that $I^{**}f=j(\indicc{X})$ implies that $\|f\| \geq 1$. We also have
$
|f(\xi)| \leq \lim_n \|\xi\| \|f_n\| = \|\xi\|
$
for every $\xi$ in $C(X, L^1G)^*$ and therefore $\|f\| \leq 1$.

\printindex
\end{document}